\documentclass[a4paper,11pt,leqno
]{article}
\usepackage{enumerate}
\usepackage[shortlabels]{enumitem}
\usepackage[utf8]{inputenc}
\usepackage[T1]{fontenc}
\usepackage{amsthm}
\usepackage[font={footnotesize}]{caption}
\usepackage{amssymb}
\usepackage{mathrsfs}
\usepackage{amsmath}
\usepackage{amsfonts}
\usepackage{mathtools}
\usepackage{graphicx}
\usepackage{float}
\usepackage{dsfont}
\usepackage[a4paper]{geometry}
\usepackage{url}
\usepackage{csquotes}
\usepackage{IEEEtrantools}
\usepackage{appendix}
\usepackage{tikz-cd}
\usepackage{multicol}
\usepackage{parcolumns}
\usepackage{stackrel}
\usepackage{comment}
\usepackage{constants}

\usepackage{hyperref}
\hypersetup{linktocpage,
  colorlinks   = true, 
  urlcolor     = blue,  
  linkcolor    = blue,
  citecolor   = blue 
}

\newtheorem{The}{Theorem}[section]

\newtheorem{Lemme}[The]{Lemma}
\newtheorem{Prop}[The]{Proposition}
\newtheorem{Cor}[The]{Corollary}
\theoremstyle{definition}

\theoremstyle{remark}
\newtheorem{Rk}[The]{Remark}

\usepackage{titlesec}
\titleformat{\subsection}[runin]{\normalfont\bfseries}{\thesubsection.}{.5em}{}[.]\titlespacing{\subsection}{0pt}{2ex plus .1ex minus .2ex}{.8em}
\titleformat{\subsubsection}[runin]{\normalfont\bfseries}{\thesubsubsection.}{.5em}{}[.]
\titlespacing{\subsubsection}{0pt}{2ex plus .1ex minus .2ex}{.8em}

\title{\large 
\textbf{
CLUSTER VOLUMES FOR THE GAUSSIAN FREE FIELD ON METRIC GRAPHS}}
\author{}
\date{}
\makeatletter
\newcommand{\E}{\mathbb{E}}

\newcommand{\R}{\mathbb{R}}
\newcommand{\Z}{\mathbb{Z}}

\newcommand{\G}{\mathcal{G}}
\newcommand{\K}{\mathcal{K}}
\newcommand{\A}{\mathcal{A}}

\renewcommand{\P}{\mathbb{P}}
\newcommand{\eps}{\varepsilon}

\newcommand{\I}{{\cal I}}
\newcommand{\V}{{\cal V}}

\renewcommand{\phi}{\varphi}
\renewcommand{\tilde}{\widetilde}
\renewcommand{\hat}{\widehat}
\renewcommand{\epsilon}{\varepsilon}
\renewcommand\theequation{\thesection.\arabic{equation}}

\@addtoreset{equation}{section}

\newconstantfamily{c}{symbol=c}

\newcommand{\PK}{\P^K}
\newcommand{\PPK}{P^K}
\newcommand{\EK}{\E^K}

\newcommand{\OP}{\overline{\P}}
\renewcommand{\OE}{\overline{\E}}
\newcommand{\OPK}{\overline{\P}^K}
\newcommand{\OEK}{\overline{\E}^K}
\newcommand{\PKK}[1]{\P^{#1}}
\newcommand{\PPKK}[1]{P^{#1}}

\newcommand{\LV}{M}

\makeatother

\usepackage{color}
\definecolor{Red}{rgb}{1,0,0} 
\definecolor{Blue}{rgb}{0,0,1}
\definecolor{Olive}{rgb}{0.41,0.55,0.13}
\definecolor{Yarok}{rgb}{0,0.5,0}
\definecolor{Green}{rgb}{0,1,0}
\definecolor{MGreen}{rgb}{0,0.8,0}
\definecolor{DGreen}{rgb}{0,0.55,0}
\definecolor{Yellow}{rgb}{1,1,0}
\definecolor{Cyan}{rgb}{0,1,1}
\definecolor{Magenta}{rgb}{1,0,1}
\definecolor{Orange}{rgb}{1,.5,0}
\definecolor{Violet}{rgb}{.5,0,.5}
\definecolor{Purple}{rgb}{.75,0,.25}
\definecolor{Brown}{rgb}{.75,.5,.25}
\definecolor{Grey}{rgb}{.7,.7,.7}
\definecolor{Black}{rgb}{0,0,0}

\def\black{\color{Black}}

\usepackage{titletoc}

\dottedcontents{section}[4em]{}{2.9em}{0.7pc}
\dottedcontents{subsection}[0em]{}{3.3em}{1pc}
\newconstantfamily{index_prob_cont_time}{
	symbol=\mathcal{T},
}

\newconstantfamily{index_prob}{
	symbol=\mathcal{N},
}
\newconstantfamily{constant}{
	symbol=\mathcal{C},
}
\newconstantfamily{small}{
	symbol=c,
}

\begin{document}
\maketitle

\begin{center}
\vspace{-1.3cm}
Alexander Drewitz$^1$, Alexis Pr\'evost$^2$ and Pierre-Fran\c cois Rodriguez$^3$ 
\end{center}

\vspace{-0.1cm}
\begin{abstract}
\begin{minipage}{0.75\textwidth}
 We study the volume of the critical clusters for the percolation of the level sets of the Gaussian free field on metric graphs. On $\Z^d$ below the upper-critical dimension $d=6$, we show that the largest such cluster in a box of side length $r$  has volume of order $r^{\frac{d+2}{2}}$, as conjectured by Werner in \cite{werner2020clusters}. This is in contrast to the mean-field regime $d>6$, where this volume is of order $r^4$. We further obtain precise asymptotic tails for the volume of the critical cluster of the origin, and a lower bound on the tail of the volume of the near-critical cluster of the origin below the upper-critical dimension. Our proof extends to any graph with polynomial volume growth and polynomial decay of the Green's function as long as the critical one-arm probability decays as the square root of the Green's function, which is satisfied in low enough dimension. 
\end{minipage}
\end{abstract}

\vspace{4cm}
\begin{flushleft}

\noindent\rule{5cm}{0.4pt} \hfill \today \\
\bigskip
\begin{multicols}{2}

$^1$Universit\"at zu K\"oln\\
Department Mathematik/Informatik \\
Weyertal 86--90 \\
50931 K\"oln, Germany. \\
\url{adrewitz@uni-koeln.de}\\[2em]

$^2$University of Geneva\\
Section of Mathematics\\
24, rue du G\'enéral Dufour\\
1211 Genève 4, Suisse.
\\\url{alexis.prevost@unige.ch}\\[2em]

\columnbreak
\thispagestyle{empty}
\bigskip
\medskip
\hfill$^3$Imperial College London\\
\hfill Department of Mathematics\\
\hfill London SW7 2AZ \\
\hfill United Kingdom\\
\hfill \url{p.rodriguez@imperial.ac.uk} 
\end{multicols}
\end{flushleft}

\newpage

\section{Introduction} \label{sec:intro}

Percolation models display a rich phenomenology near their critical point, which depends in general on parameters such as the underlying spatial dimension of the model and is notoriously difficult to describe rigorously. In this context, among the quantities of central interest are the so-called \emph{critical exponents}, which are expected to universally characterize the critical and near-critical behavior of these models. In the well-studied case of Bernoulli percolation, significant progress has been made on certain two-dimensional lattices \cite{MR879034,Sm-01,SmWe-01,MR1887622}, as well as within the mean-field regime in high enough dimensions \cite{Bar-Aiz91,HarSla90,MR1431856,KozNac11,10.1214/17-EJP56}, see also \cite{duminilcopin2024alternativeapproachmeanfieldbehaviour} for very recent developments. However, the understanding of intermediate dimensions, which lie between these two regimes,  remains largely an open problem.

In such dimensions, recent advances in the investigation of critical behavior have emerged from studying a bond percolation model linked to the Gaussian free field, which belongs to a different universality class than Bernoulli percolation. This model, introduced in \cite{MR3502602}, is the metric-graph version of the discrete Gaussian free field excursion sets, whose study goes back to \cite{MR914444}, and which was more recently re-initiated in \cite{MR3053773}; in fact the two models are presumably in the same universality class \cite{chalhoub2024universality}. The metric graph version has provided significant insights in the mean-field regime \cite{werner2020clusters,cai2023onearm,ganguly2024ant}, and -- notably -- also in the challenging intermediate dimensions \cite{DiWi,DrePreRod5,DrePreRod3,DrePreRod8,DrePreRod9,cai2024onearm}, which lie below the upper critical dimension. The model also connects with other frameworks like random interlacements or loop soups, and exhibits a continuous transition with scaling near the critical point \cite{DrePreRod5}. These properties are crucial for understanding the (near-)critical cluster volumes, on which we focus in this article. Our results contribute to the understanding of \cite[Conjecture~A and C]{werner2020clusters}.

For simplicity, we focus on $\Z^d$, $d \geq 3$, in this introduction but our results are in fact more general; see \eqref{eq:critvol-toblerone0}--\eqref{eq:critvol-toblerone} for an example, and Section~\ref{sec:mainresults} for our results in their most general form. We denote by $\phi$ the Gaussian free field on $\Z^d$, that is the centered Gaussian field whose covariance function is the Green function $g(\cdot,\cdot)$ on $\Z^d$. The percolation model associated to $\phi$ we will be interested in can be defined as follows: conditionally on $\phi$, open each edge $\{x,y\}$ of $\Z^d$ independently with probability
\begin{equation}
\label{eq:disrete_model}
\begin{split} 
1-\exp\big(-2(\phi_x-a)_+(\phi_y-a)_+\big).
\end{split}
\end{equation}
Here $a\in{\R}$ is the percolation parameter, and the above procedure defines a percolation model at level $a$, which is naturally increasing in $a$. Alternatively, this percolation model can be described via the metric graph associated to $\Z^d$, and we refer to Section~\ref{sec:mainresults} for details. We denote by $\K^a$ the set of vertices in $\Z^d$ which are connected to the origin by a path of open edges at level $a$. It is known \cite{MR3502602,DrePreRod3} that $a=0$ is the critical level for percolation of $\K^a$. That is, $\P$-a.s., $\K^a$ is a.s.\ finite when $a\geq 0$, and is  infinite with positive probability when $a<0$. Moreover, the cluster capacity is an integrable quantity, and for later reference we note 
that (see \cite[Corollary 1.3]{DrePreRod5}) 
\begin{equation}\label{eq:cap-tail}
\P\big( \text{cap}({\mathcal{K}}^0) > t \big) \sim \frac{1}{\pi \sqrt{gt}} \text{ as } t \to \infty
\end{equation}
where $g=g(0,0)$ and $a\sim b$ means $\frac ab \to 1$ in the given limit. For $r\geq1$ and $a\in{\R}$, we denote by $|\K^a|$ the cardinality of the cluster of the origin $\K^a$. Furthermore, the connected components at level $a$ in a ball of radius $r$ around $0$ can be ordered by cardinality, and we denote by $\LV_r^a$ the cardinality of the largest such cluster. We refer to \eqref{eq:vol} and \eqref{eq:defVra} below for precise definitions. The tail of the \emph{cluster size}, or \emph{volume} functional $|\mathcal{K}^a|$, as well as the typical behavior of the \emph{largest cluster size} $\LV^a_r$, will be the main objects of interest in this article; see Remark~\ref{R:vol-forms} regarding other (natural) choices of volume functionals. Our first result concerns the critical case $a=0$  below the upper critical dimension $d=6$.

\begin{The}\label{T:critvol-Zd}
    For $d\in \{ 3,4,5 \}$, there exist $c=c(d), C=C(d) \in (0,\infty)$ such that for all $n \geq 1$,
    \begin{equation}
        \label{eq:critvol-Zd}
      cn^{-\frac{d-2}{d+2}} \leq  \P(|\mathcal{K}^0| \geq n) \leq  Cn^{-\frac{d-2}{d+2}},
    \end{equation}
    and for all $r,t\geq1$ with $r^{\frac{d+2}{2}}\geq Ct$,
    \begin{equation}
    \label{eq:critlargestvol-Zd2}
    \begin{split} 
        \P\big((1/t)r^{\frac{d+2}{2}}\leq \LV^0_r\leq tr^{\frac{d+2}{2}}\big)
            \geq 1- Ct^{-c}.
    \end{split}
    \end{equation}
\end{The}

The inequalities in~\eqref{eq:critvol-Zd} follow as a special case of Theorem~\ref{thm:volLB} and \eqref{eq:boundq-new}--\eqref{eq:critvol-ub}, with the choices $a=0$ and $\alpha = d = \nu+2$ therein. With the same choices, \eqref{eq:critlargestvol-Zd2} follows by combining Proposition~\ref{pro:upperMRa} (see below \eqref{eq:upperMR0}) and Theorem~\ref{the:mainMra} (see \eqref{eq:tailEstlargestVolgena=0}). Note that \eqref{eq:critvol-Zd} entails that the (fractional) moments $\E[|\mathcal{K}^0|^{\gamma}]$ of the critical cluster volume are finite if and only if $\gamma < \frac{d-2}{d+2}$. We also refer to Remark~\ref{rk:final},\ref{rk:volumeinball} for lower bounds on the expected  cardinality of $\K^0$ in a ball, which are presumably sharp.

We now briefly comment on the  bounds corresponding to \eqref{eq:critvol-Zd} for mean-field values of $d$.  
For $d=6$, one knows that for all $n \geq 1$,
\begin{equation}
    \label{eq:critvol-Z6}
   cn^{-\frac{1}{2}} \leq  \P(|\mathcal{K}^0| \geq n) \leq  C \tilde q(n) n^{-\frac{1}{2}},
\end{equation}
where $\tilde q(n)=1 \vee \exp\{ C \sqrt{\log n} \log \log n\}$; the lower bound is obtained from \eqref{eq:cap-tail} and sub-additivity of the capacity, by which $\text{cap}({\mathcal{K}}^0) \leq C|\mathcal{K}^0|$; see \eqref{eq:subadd} below. The upper bound is a special case of \eqref{eq:critvol-ub} below (which follows the line of argument used to prove \cite[Corollary 1.6]{DrePreRod5}) and the result of \cite{cai2024onearm}. For all $d \geq 7$, \eqref{eq:critvol-Z6} remains true but with $\tilde q(n)=1$.
The lower bound is obtained as above as already noted in \cite{cai2023onearm}, and the upper bound follows from \cite[Theorem~1.2]{cai2023onearm}. Unlike when $d\geq 7$, the lower bound in \eqref{eq:critvol-Z6} for $d=6$ may no longer be sharp. 

The inequality \eqref{eq:critlargestvol-Zd2} follows from  Proposition~\ref{pro:upperMRa}, Theorem~\ref{the:mainMra} and \eqref{eq:boundq-new}. Concerning $\LV_r^0$ in higher dimensions, we have when $d=6$
\begin{equation}
\label{eq:critlargestvol-Zd2meanfield}
\begin{split} 
    \P\big((1/t)r^{4}\tilde{q}(r)^{-1}\leq \LV^0_r\leq tr^{4}\hat{q}(r)\big)
        \geq 1- Ct^{-c}\text{ if }r^4\tilde{q}(r)^{-1}\geq Ct,
\end{split}
\end{equation}
where $\tilde{q}$ is as below \eqref{eq:critvol-Z6}, and $\hat{q}(r)=1$ for all $r\geq1$. The upper bound follows from Proposition~\ref{pro:upperMRa}, and the lower bound from \eqref{eq:qPsiBd}, Proposition~\ref{pro:meanfieldlower} and the results from \cite{cai2023onearm}. In the mean-field regime $d\geq7$, \eqref{eq:critlargestvol-Zd2meanfield} is still verified but with $\tilde{q}(r)=1$  and $\hat{q}(r)=\ln(r)$ for all $r\geq1$. The lower bound is still a consequence of Proposition~\ref{pro:meanfieldlower} and \cite{cai2023onearm}, whereas the upper bound follows from \cite[Lemma~5.3]{cai2023onearm}, see also \cite[Proposition~3]{werner2020clusters}.

Noteworthily, the tail behavior \eqref{eq:cap-tail} of the capacity observable, which holds on a large class of graphs, yields increasingly sharp information on the \emph{radius} of a cluster in \text{low} dimensions (see in particular \cite[(1.22) or Theorem 1.4(i)]{DrePreRod5}), whereas it yields sharp information on cluster \emph{volume} in \emph{high} dimensions (see \eqref{eq:critvol-Z6}, in particular when $d\geq 7$). This is related to the fact that, in \emph{low} dimensions, sets of a given diameter become more and more indistinguishable in terms of capacity, whereas the capacity, a sub-additive quantity, becomes increasingly additive (and thus close to the volume) in high dimensions.

Let us now summarize Theorem~\ref{T:critvol-Zd} for $3 \leq d\leq 5$ as well as the simple consequences \eqref{eq:critvol-Z6} and \eqref{eq:critlargestvol-Zd2meanfield}  from  \cite{werner2020clusters,cai2023onearm,cai2024onearm} for $d\geq6$ (note that the functions $\tilde{q}$ and $\hat{q}$ appearing therein are subpolynomial), in terms of the critical exponent $\delta$ for the tail of the cluster size and $d_f$ for the largest cluster size.

\begin{Cor}
\label{cor:deltaanddf}
On $\Z^d$ for $d\geq3$,
\begin{equation}
\label{eq:delta}
\begin{split} 
\delta\stackrel{\textnormal{def.}}{=}\lim\limits_{n\rightarrow\infty}\frac{\ln(n)}{\ln(\P(|\K^0|\geq n))}=
\begin{cases}
    \frac{d+2}{d-2}&\text{ if }d\leq 6,
    \\2&\text{ if }d\geq6,
\end{cases}
\end{split}
\end{equation}
and $\P$-a.s.
\begin{equation}
\label{eq:df}
\begin{split} 
d_f\stackrel{\textnormal{def.}}{=}\lim\limits_{r\rightarrow\infty}\frac{\ln(\LV_r^0)}{\ln(r)}=
\begin{cases}
    \frac{d+2}{2}&\text{ if }d\leq 6,
    \\4&\text{ if }d\geq6.
\end{cases}
\end{split}
\end{equation}
\end{Cor}

Note that even the mere existence of the exponents $\delta$ and $d_f$ in \eqref{eq:delta} and \eqref{eq:df} was previously not known when $ 3 \leq d\leq5$. The value of $\delta$ from \eqref{eq:delta} when $d\leq 6$ was first conjectured in \cite[Table~1]{DrePreRod5} and identified as a question \lq of great interest\rq\ in \cite[Remark~1.4]{cai2023onearm}. The value of the fractal dimension $d_f$, cf.\ \cite{StAh-18}, in \eqref{eq:df} when $d\leq 6$ has been conjectured in \cite[Conjecture~A and C]{werner2020clusters}. We refer to Remark~\ref{rk:final},\ref{rk:wernerconjecture} for more details as to why \eqref{eq:critlargestvol-Zd2} partially solves further conjectures from \cite{werner2020clusters}. 
 Alternatively, one could define $d_f$ via the average of $\LV^0_r$, that is $\E\big[\LV^0_r\big]=r^{d_f+o(1)}$, which would still satisfy $d_f=(d+2)/2$ when $d\leq 6$ by \eqref{eq:critlargestvol-Zd2} and \eqref{eq:upperMR0}. When $d\geq6$, the critical exponents $\delta$ and $d_f$ from \eqref{eq:delta} and \eqref{eq:df} correspond to the mean-field values identified for high dimensional independent percolation in \cite[Corollary~4.2]{Bar-Aiz91},  \cite[Theorem~5]{MR1431856} and \cite[Theorem~1.1]{HarSla90}, see also  Corollary~5.2 and Theorems~9.2, 11.4 and 13.22 in \cite{HeHo-17}. 
 
 The fact that many critical exponents have a simple form for $d\leq6$, and then take their mean-field value for $d\geq6$ was first conjectured in \cite{DrePreRod5}, and then proved in the case of the critical one-arm exponent $\rho$ in \cite{DiWi,cai2023onearm,DrePreRod8,cai2024onearm}. Corollary~\ref{cor:deltaanddf} thus builds on these results  (except for $\delta$ and the upper bound on $d_f$ in dimension $d\geq7$ which was already obtained in \cite{werner2020clusters,cai2023onearm})  to prove that this is also the case for the critical volume exponents $\delta$ and $d_f$.

Our results on critical cluster volumes are actually far more general than Theorem~\ref{T:critvol-Zd}. To illustrate this, consider for example graph obtained as a Cartesian product of $\Z^2$ with the discrete skeleton of the Sierpinski gasket; cf.~\cite[Fig.~1]{DrePreRod2} (the product with $\Z^2$ ensures this graph is transient). Then the results of Section~\ref{sec:mainresults} yield that for this graph, for all $n,r\geq 1$,
\begin{equation}\label{eq:critvol-toblerone0}
cn^{-\frac{1}{\delta}} \leq  \P(|\mathcal{K}^0| \geq n) \leq  Cn^{-\frac{1}{\delta}}\text{ and }cr^{d_f} \leq  \E\big[\LV^0_r\big] \leq  Cr^{d_f},
\end{equation}
with
\begin{equation}
    \label{eq:critvol-toblerone}
    \delta =\frac{2\log (3) + 4 \log (5)}{2\log( 3)}= 3.929947...\text{ and }d_f=\frac{\log (3) + 2 \log (5)}{2\log (2)}=2.533927...
\end{equation}
 In particular, \eqref{eq:critvol-toblerone0}--\eqref{eq:critvol-toblerone} illustrate that our results remain valid in cases (unlike \eqref{eq:critvol-Zd}) where the exponents $\delta$ and $d_f$ are not at all nice (algebraic) numbers.  
In order to obtain \eqref{eq:critvol-toblerone0}, one needs to combine \eqref{eq:boundq-new}, \eqref{eq:critvol-ub}, \eqref{eq:tailEstVolgen}, \eqref{eq:upperMR0} and \eqref{eq:tailEstlargestVolgen} for $a=0$, $\alpha=\frac{\log(3)+\log(5)}{\log(2)}$ and $\nu=\frac{\log(3)}{\log(2)}$ 
(see \cite[Proposition~3.5]{DrePreRod2} and \cite{MR1378848}), which satisfy $1<\nu<\alpha/2$ (and hence one can apply the results from \cite{DrePreRod8}). This is but one example. Our results do in fact hold under mild assumptions on the underlying graph, comprising mainly a suitable polynomial volume growth assumption for balls, and the polynomial decay of the Green function for the Laplacian; see Section~\ref{sec:mainresults} for details (the assumptions just mentioned correspond to conditions~\eqref{eq:intro_sizeball} and~\eqref{eq:intro_Green} appearing there). We refer to the introduction of \cite{DrePreRod2}, in particular (1.4) therein, for a host of other examples to which our results apply. These include, among others, all Cayley graphs of finitely generated groups with suitable volume growth.

 \bigskip
We now present our results in the off-critical case, formulated in the benchmark cases of $\Z^d$, $d \geq 3$, for the sake of simplicity. These results concern $|\mathcal{K}^a|$ when $a \neq 0$. We first focus on the tail of the (truncated) cluster size distribution, which we show in Proposition~\ref{pro:easyoffcriticalbounds} satisfy
\begin{equation}
\label{eq:offcritvol-n}
\exp\big\{ - C(a^2\vee1) n^{\frac{d-2}{d}}  \big\} \leq  \P(n \leq |\mathcal{K}^a| < \infty) \leq \exp\big\{ - ca^2n^{\frac{d-2}{d}}  \big\} ,
\end{equation}
 for all $d\geq3$, $a\in{\R}$ and $n\geq1$. The proof of the inequalities \eqref{eq:offcritvol-n} on the cables
follows from by now relatively standard arguments. One can dispense with the truncation $|\mathcal{K}^a| < \infty$ in \eqref{eq:offcritvol-n} when $a>0$ since this occurs $\P$-a.s. We can also prove bounds similar to \eqref{eq:offcritvol-n} when replacing $|\mathcal{K}^a|$ by $\LV^a_r$, but only for $r$ large enough (depending on $a$ and $n$), see Remark~\ref{rk:final},\ref{rk:volumeinball}. The inequalities \eqref{eq:offcritvol-n} exhibit a very different tail behavior for the cluster size distribution than in the case of Bernoulli percolation for instance. The decay in \eqref{eq:offcritvol-n} is sub-exponential (and identical) in both sub-critical ($a>0$) and super-critical $(a<0)$ regimes, whereas for Bernoulli percolation, the decay is exponential in the sub-critical phase \cite{MR0633715}, and decays exponentially in $n^{\frac{d-1}{d}}$ in the supercritical phase \cite{MR0594824,MR1055419}; see also \cite{MR4734554}  for related results. Intuitively, the sub-critical Bernoulli cluster is much more \emph{amorphous} than $\mathcal{K}^a$, $a>0$, where long-range correlations are present. The discrepancy in the super-critical direction is owed to a surface-order (rather than capacitary) cost needed to isolate a large volume from the ambient infinite cluster.

The bounds \eqref{eq:offcritvol-n} are clearly optimal up to constants when $|a|\geq1$, i.e.~far away from criticality. Our main off-critical result yields an improvement over the first bound in \eqref{eq:offcritvol-n} in low dimensions by providing a better lower bound when $|a|\leq 1$, and in particular close to the critical point.

\begin{The}
\label{T:offcritvol-Zd-LB}
For each $d\in \{ 3,4,5 \}$, there exist constants $c = c(d), C = C(d)\in (0,\infty)$ such that, for all $a \in [-1,1]$  and all $n \geq 1$, 
\begin{equation} \label{eq:offcritvol-Zd-LB}
\P\big( n \leq |\mathcal K^a| <\infty \big) \ge
  c\P\big( n \leq |\mathcal K^0| <\infty \big)  \exp \big\{ -C |a|^{\frac{d+2}{d}}n^\frac{d-2}{d} \big\}.  
\end{equation}
\end{The}

The combination of critical and off-critical lower bounds supplied by Theorem~\ref{T:critvol-Zd} and Theorem~\ref{T:offcritvol-Zd-LB} below actually covers all values of $(n,a)$; cf.~Theorem~\ref{thm:volLB} below, of which Theorem~\ref{T:offcritvol-Zd-LB} is a special case in view of \eqref{eq:boundq-new}.
We believe that the quantitative lower bound in \eqref{eq:offcritvol-Zd-LB} is actually sharp, see the discussion below  Corollary~\ref{C:offcritvol-mom} for details, and it would be very interesting to derive corresponding upper bounds. We refer to Remark~\ref{rk:final},\ref{rk:optimalboundmeanfield} for a conjecture for similar sharp bounds in the mean-field regime $d\geq7$, and to Remark~\ref{rk:final},\ref{rk:d=6} for a result akin to \eqref{eq:offcritvol-Zd-LB} but with additional subpolynomial corrections in the exponential in the critical dimension $d=6$. One can also derive presumably sharp off-critical lower bound for the largest cluster volume $\LV_r^a$ which generalize the critical result \eqref{eq:critlargestvol-Zd2}, and we refer to Proposition~\ref{pro:upperMRa}, Theorem~\ref{the:mainMra} and the discussion below for details.

The above results can be applied to infer corresponding lower bounds on moments of the cluster size close to the critical point.

\begin{Cor}
\label{C:offcritvol-Zd-mom}
    For each $d\in \{ 3,4,5 \}$ and each $k \geq 1$, there exists $c=c(d,k) \in (0,\infty)$ such that for all $a \in [-1,1]$, 
    \begin{align} \label{eq:offcritvol-Zd-mom}
    \E\big[|\mathcal K^{a} |^k\cdot 1\{|\mathcal{K}^a|<\infty\}\big]  \ge c|a|^{1-k\frac{d+2}{d-2}}.
\end{align}
\end{Cor}
We refer to Corollary~\ref{C:offcritvol-mom} for a more general result, by which \eqref{eq:offcritvol-Zd-mom} becomes a special case of \eqref{eq:offcritvol-mom2} (with the choices $\alpha=d=\nu+2$) on account of \eqref{eq:boundq-new}.  The bound \eqref{eq:offcritvol-Zd-mom} on the first moment ($k=1$) is already known on $\Z^3$ together with a matching upper bound, see \cite[Corollary~1.5]{DrePreRod5}, where it was obtained by different means. In dimension four, an upper bound matching \eqref{eq:offcritvol-Zd-mom} for $k=1$ up to logarithmic corrections follows from \cite[Theorem~1.2]{DrePreRod8}, see below Corollary~1.4 therein for details, and one readily deduces  by combining \eqref{eq:offcritvol-Zd-mom} for $k=1$ with the previous upper bound from \cite{DrePreRod8} that the critical exponent  $\gamma$ as defined in \cite[(1.16)]{DrePreRod8} satisfies  $\gamma=2$ on $\Z^4$. 
The behavior of higher moments is conventionally parametrized in terms of an exponent $\Delta >0$ capturing the polynomial divergence of ratios of moments as in \eqref{eq:offcritvol-Zd-mom} for (integer) $k$ and $k+1$ as $a \to 0$, which conjecturally does not to depend on the value of $k$. If complemented by matching upper bounds, the results of Corollary~\ref{C:offcritvol-Zd-mom} would validate the conjectured value $\Delta=\frac{2d}{d-2}-1$ from \cite[Table I]{DrePreRod5}. This circumstance is also among the reasons why we expect the estimate \eqref{eq:offcritvol-Zd-LB} to be sharp in dimensions $d\in{\{3,4,5\}}$. We hope to return to  all of this in future work. In the mean-field regime $d\geq7$, the bound \eqref{eq:offcritvol-Zd-mom} is actually not true anymore, and we refer to \eqref{eq:offcritvol-mom1} and the discussion below for a better bound which is expected to be sharp in this regime. 

\begin{figure}
    \centering
    \includegraphics[width=0.70\linewidth]{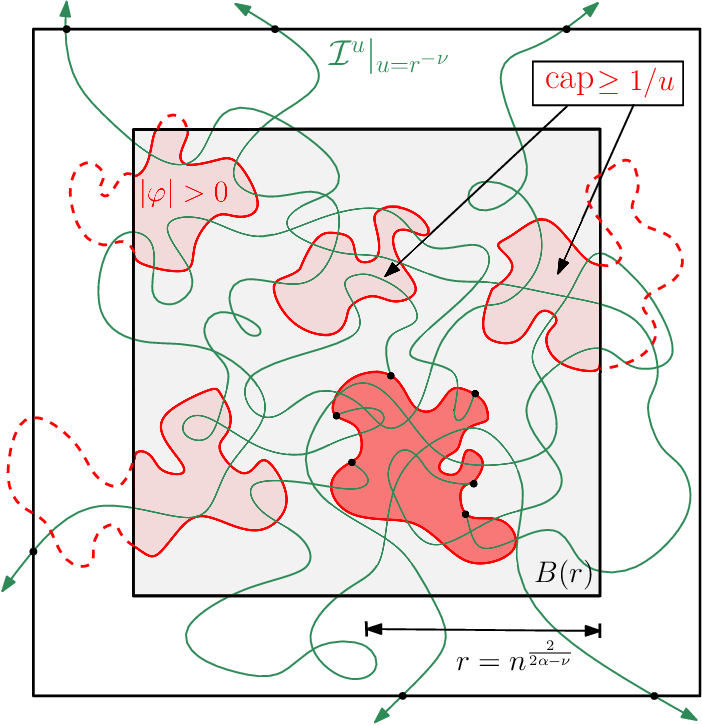}
    \caption{Constructing a large \emph{critical} cluster (at height $a=0$) comprising $n$ points (with $\alpha=d=\nu+2$ in the case of  $\Z^d$). The cluster is roughly obtained as the union of all the macroscopic sign clusters (light+dark red) inside a box of radius $r$ hit by an independent interlacement (green) at level $u=r^{-\nu}$, for the above fine-tuned scale $r=r(n)$. More precisely, and to ensure that this set is connected, one retains only (the forward part of) all the interlacement trajectories that intersect the sign cluster with largest capacity inside $B(r)$ (dark red), which acts as a \emph{hub}. The trajectories are discarded upon exiting a slightly larger box so as for the event to remain local. The picture remains pertinent when constructing a large \emph{off-critical} cluster (at height $a\neq 0$), for a different choice of $r=r(a,n)$ given by \eqref{eq:heur2}  and with $u\asymp |a|^2$. The cross-over between the two choices for $r$ occurs precisely when the exponential term appearing in Theorem~\ref{T:offcritvol-Zd-LB} is of unit order.    }
    \label{fig:vol-gff}
\end{figure}

\bigskip
 We now comment on the proofs of our main results, Theorems~\ref{T:critvol-Zd} and~\ref{T:offcritvol-Zd-LB}.  To avoid being overly technical, we will not attempt to give a full overview of the proofs here but rather aim to give a flavor of the mechanisms lurking behind the various tails in \eqref{eq:critvol-Zd}, \eqref{eq:critlargestvol-Zd2} and \eqref{eq:offcritvol-Zd-LB}.
We refer to Figure~\ref{fig:vol-gff} for visual aid accompanying the following discussion.

 We start with the lower bounds in \eqref{eq:critvol-Zd} and \eqref{eq:critlargestvol-Zd2}, the upper bounds being simple consequences of, respectively, \cite{cai2024onearm} and the reasoning from \cite[Corollary 1.6]{DrePreRod5} in the case of \eqref{eq:critvol-Zd}, and of \cite[Proposition~5.2]{MR3502602} and the Cauchy-Schwarz inequality, see Proposition~\ref{pro:upperMRa}, in the case of \eqref{eq:critlargestvol-Zd2}. In low dimensions ($d=3,4,5$), two opposite forces at play are the following. On the one hand, rather than witnessing a (large) cluster at height $a=0$ inside some box $B=B(r)$, for an appropriate choice of $r$ that  will be fixed later, one can typically afford to consider a height $-b<0$ instead, as long as $ r \leq \xi(b)\stackrel{\textnormal{def.}}{=}b^{-2/(d-2)}$; one way to see this is that shifting the field by $-b$ inside that box will occasion an entropic (Cameron-Martin type) cost of order $e^{-cb^2 \text{cap}(B)} \geq c$, see \eqref{eq:entropy} and \eqref{eq:capBallBd}. On the other hand, $\xi(b)$ is a lower bound for a scale $r$ that the infinite cluster in $\{ \varphi \geq -b\}$ (for $b>0$) can penetrate; this can be seen via appropriate isomorphism theorem, cf.~\eqref{eq:isom} below. Heuristically, the field ``chooses'' $r\asymp \xi(b)$ as the optimal scale at which these two effects are simultaneously in force, for a value of $b$ such that the desired volume constraint $n$ in \eqref{eq:critvol-Zd} is met. Since the infinite cluster has a significant presence at this scale, this roughly amounts to tuning $b$ (for given $n$) such that
\begin{equation}\label{eq:heur1}
n=\theta(b) \xi(b)^d \asymp |b|^{-\frac{d+2}{d-2}},
\end{equation}
where $\theta(b) \asymp |b|$ ($b<0$) is the density of the infinite cluster, cf.~\cite[Corollary 1.2]{DrePreRod5}. The choice for $b$ in \eqref{eq:heur1} leads to the box size $r= n^{-\frac2{d+2}}$($\asymp\xi(b)$) in Figure~\ref{fig:vol-gff}. Although  this critical scale is chosen so that the infinite cluster has a significant presence (i.e.~size $n$) in \eqref{eq:heur1}, it is however not clear at all whether it mainly consists of one giant connected component at scale $\xi(b)$ containing most of its $n$ points, or if it splits into several much smaller connected components in $B(\xi(b))$. The latter will in fact occur on $\Z^d$, $d\geq7$ , see Remark~\ref{rk:final},\ref{rk:boundinanydimension}, and proving that this is not the case in dimension $d\in{\{3,4,5\}}$ thus requires further arguments, which are depicted in Figure~\ref{fig:vol-gff} and to which we return below.  Foregoing this issue for the time being, \eqref{eq:heur1} implies that the probability that $0$ is in the infinite cluster at level $-b$ is of order $\theta(b)\asymp |b|\asymp n^{-\frac{d-2}{d+2}}$, yielding \eqref{eq:critvol-Zd} since the infinite cluster at level $-b$, and hence $0$ by the previous choice of $r$ and $b$, has a connected component of size of order $n$. A variation of this reasoning accounts for the typical size $M_r^0$ of the largest cluster inside a box of (given) radius $r$  appearing in \eqref{eq:critlargestvol-Zd2}. Indeed, $r$ is now fixed and one can effectively afford to look for a large cluster at level $b(<0)$ such that $|b|= r^{\frac{2-d}{2}}$ (so that $r=\xi(b)$), at which the infinite cluster will contribute $\theta(b) r^d\asymp |b|r^d= r^{\frac{2+d}2}$ points.

We now discuss \eqref{eq:offcritvol-Zd-LB}, and in so doing also refine the above heuristic picture. For a given level $a$ with $|a|\ll 1$, by the same principle as in
\eqref{eq:heur1} (with $b=a$), the correct length scale $r$ to look for a cluster of volume $n$ is given by solving
\begin{equation}\label{eq:heur2}
n=\theta(a)r^d\asymp|a|r^d, \quad (|a|\leq 1),
\end{equation}
that is $r=(n/|a|)^{1/d}$. The exponential decay in \eqref{eq:offcritvol-Zd-LB} then arises as an entropic cost  \eqref{eq:entropy} of the order $e^{-ca^2 r^{d-2}}$ for the choice of $r$ in \eqref{eq:heur2}. However, while this is good enough for the purpose of creating a large cluster (as needed for $M_r^a$ to be large), connecting it to the origin requires not only ensuring that the cluster has large volume but also a large capacity (see the end of Section~\ref{sec:pre} as to why; roughly speaking, the connection to the cluster of $0$ happens through an interlacement trajectory at intensity of order $a^2$). Furthermore, additional complications will be caused by the probability appearing in the exponential of the entropy bound \eqref{eq:entropy}. We refer to the end of Section~\ref{sec:pre} for more on this.

It remains to explain how to obtain the simultaneous occurrence of \emph{both} large volume, in the sense of \eqref{eq:heur2},  and capacity of a \emph{single} cluster  among all the possible clusters of the infinite connected component at scale $r$ (with sizeable probability even if $r$ is of the same order as the critical scale $\xi$) in dimension $d\in{\{3,4,5\}}$, which is the object of the central Proposition~\ref{pro:manyMeso}. It crucially takes advantage of the (signed) isomorphism \eqref{eq:isom} with interlacements, as depicted schematically in Fig.~\ref{fig:vol-gff}, but now with $r$ as given below \eqref{eq:heur2}.  In essence, we consider the critical cluster with the largest capacity in $B(r)$ (in darker red) to ensure a capacity of the same order as that of $B(r)$, which is possible for $d\in{\{3,4,5\}}$ using the result from \cite{cai2024onearm}, see Lemma~\ref{L:cap-LB} for details. We then add all the interlacements trajectories (in green) which intersect this giant cluster, as well as all the critical clusters (in red) which intersect those trajectories, which validate the volume condition. The isomorphism then asserts that in this picture both red and green can be viewed as part of the large cluster of interest, which is indeed connected at scale $r$ by definition. This is enough to obtain an analogue of \eqref{eq:offcritvol-Zd-LB} for the largest cluster (cf.~Theorem~\ref{the:mainMra} below). To make the cluster \emph{of $0$} large, an additional cost of order $ \xi^{-(d-2)/2} = |a|$ (with $\xi=\xi(a)$) is incurred in order to make sure the cluster of $0$ in $B(\xi)$ at height $a$ has large capacity (not depicted in Fig.~\ref{fig:vol-gff}). This additional cost can essentially be viewed as governed by the tail \eqref{eq:cap-tail} of the \emph{critical} capacity, if one forces it to be of order $\xi^{d-2}$; see Lemma~\ref{lem:capbox} as to how this rigorously transfers to height $a \neq 0$. As it turns out, after paying the price $|a|$ to ``blow it up,'' the cluster of the origin connects to the ambient large cluster with probability of order one uniformly in $a$. In the critical case \eqref{eq:critvol-Zd} (with $b$ in place of $a$), this additional cost precisely governs the order of magnitude ($=|b|$) of the probability in question in the discussion following \eqref{eq:heur1}.

\bigskip
We conclude this introduction by discussing consequences of the above results for a cousin model, the percolation problem for excursion sets of the (discrete) Gaussian free field on $G$, see for instance \cite{MR914444,MR3053773}. The critical parameter for this model is traditionally denoted $h_*=h_*(d)$ and one knows \cite{MR3053773,MR3339867,DrePreRod} that $h_*(d)>0$ for all $d \geq 3$, which reflects the fact that it is harder to percolate in the above model. Indeed, $\mathcal{K}^a \cap G$ is a subset of the corresponding discrete excursion cluster of $0$ (above level $a$), hence all lower bounds derived above for $a>0$ (for which the cluster size observable requires no truncation, thus being monotone) remain valid for discrete excursion sets. In particular, picking $a>h_*$, the lower bound in \eqref{eq:offcritvol-n} yields a matching bound for the upper bound recently derived in \cite[Corollary 1.4]{zbMATH07605989}.  In fact, the lower bound of \eqref{eq:offcritvol-n} holds in much greater generality, see \eqref{eq:offcritvol-n} below, and remains pertinent on a large class of graphs $\G$ for which one knows that $h_*(\mathcal{G})>0$, cf.~\cite{DrePreRod2}. Using the isomorphism with random interlacements, see \cite[Theorem~3]{MR3502602}, one can also prove a similar lower bound for the vacant set $\V^u$ of random interlacements for all $u>0$, and in particular for all subcritical values of $u$ where this lower bound is expected to be sharp up to constants depending on $u$.

\bigskip

Let us describe the organization of the article. Section~\ref{sec:mainresults} presents the setup of general metric graphs that we will consider throughout the article, and in particular introduces our standing assumptions \eqref{eq:ellipticity}, \eqref{eq:intro_sizeball} and \eqref{eq:intro_Green}. We also state our main results in this setup, from which all the results on $\Z^d$ from this introduction can be deduced. 

Section~\ref{sec:pre} contains various preliminary results which will be useful throughout the rest of this article. Their interest is illustrated in the proof of Proposition~\ref{pro:easyoffcriticalbounds}, which corresponds to the off-critical bounds \eqref{eq:offcritvol-n}, and already conveys many of the ideas behind the proofs of our main theorems. 

Section~\ref{sec:volneg} is centered around Proposition~\ref{pro:manyMeso} which controls the volume $\LV_r^{-a}$ of the largest cluster in a ball of size $r$ at level $-a$, for $a>0$ and $r\geq1$, and whose proof is illustrated in Figure~\ref{fig:vol-gff}. This serves as a central ingredient in Section~\ref{sec:den}, together with the entropy bound \eqref{eq:entropy}, where we conclude the proof of our all our main results. Finally, the appendix contains a short proof of the lower bound on $M_r^0$ in the mean-field regime from \eqref{eq:critlargestvol-Zd2meanfield}, generalized to our setup from Section~\ref{sec:mainresults}.

\medskip

We now specify our policy with constants.
Throughout, $c,c',\tilde{c}, \tilde{c}',C,C',\dots$ denote generic positive and finite constants that change from place to place and may depend implicitly on the parameters $\alpha$ and $\nu$ in  \eqref{eq:intro_Green}, \eqref{eq:intro_sizeball}, whenever these conditions are assumed to hold (they also implicitly depend on the specific values of the constants $c_0$, $\Cr{c:alpha-lb},\Cr{c:alpha-ub},  \Cr{c:green-lb}, \Cr{c:green-ub}$ appearing in \eqref{eq:ellipticity}, \eqref{eq:intro_Green}, \eqref{eq:intro_sizeball}, which we assume fixed once and for all). If these constants depend on additional parameters, this will be stated explicitly. Numbered constants $c_0,c_1,c_2,C_0,C_1, C_2,\dots$ are defined upon their first appearances in the text and remain fixed until the end.

\section{Setup and main results}\label{sec:mainresults}

We now introduce the general setup considered in this article, and then state our main result within this framework. We consider a countable weighted graph $\G = (G, \lambda)$. That is, $G$ is a countable set and $\lambda_{x,y}=\lambda_{y,x} \geq 0$ are non-negative symmetric weights. Two points $x,y \in G$ are neighbors if and only if $\lambda_{x,y}>0$. In this case, we refer to the unordered pair of vertices $\{x,y\}$ as an edge, and write $x\sim y$ if and only if $\{x,y\}$ is an edge. We will always assume that the associated graph is locally finite. We assume that $\G$ is transient for the \emph{random walk} on $\G$, which is defined as the continuous time Markov chain $X$ on $G$ with generator
$Lf(x) = \frac{1}{\lambda_x} \sum_{y \in G} \lambda_{x,y} (f(y) - f(x)),$
for suitable $f: G \to \R$, and where $\lambda_x \stackrel{\text{def.}}{=} \sum_{y \sim x} \lambda_{x,y}$. We write $P_x$ for its law when started in $x \in G.$ Its Green function is defined as 
\begin{align} \label{eq:green}
g(x,y) = \frac{1}{\lambda_y}\int_{0}^\infty P_x(X_t =y)\, {\rm  d}t, \quad x,y \in G.
\end{align}
We will work under the following three assumptions on $\G$. The first is a standard ellipticity assumption on the weight function $\lambda$, which requires that for some constant $c_0 \in (0,\infty)$,
\begin{align}
&\label{eq:ellipticity} \tag{$p_0$}
\ \, \lambda_{x,y}/\lambda_x\geq c_0 \text{ for all $x,y \in G$ such that $\lambda_{x,y}>0$}. 
\end{align}
 We think of $\lambda_{\cdot}$ as a measure on $G$ and write $\lambda(A)=\sum_{x\in A}\lambda_x$ for $A\subset G$. The remaining two assumptions on $\G$ depend on the choice of a metric $d(\cdot,\cdot)$ defined on $G$. For many cases of interest, one can afford to simply choose $d= d_{\text{gr}}$, the graph distance on $\mathcal{G}$, i.e.~$d_{\text{gr}}(x,y)=1$ if and only if $\lambda_{x,y}>0$ (extended to a geodesic distance on $G$); we refer to \cite{DrePreRod2} for an extensive discussion of settings which may require different choices of $d.$ We write $B(x,r)=\{ y \in G : d(x,y) \leq r\}$, $x \in G$, $r>0$, for the discrete balls in the distance $d$. For a subset $K\subset G$ we write $|K|$ for its cardinality.

We will require the weighted graph $\G$ to be $\alpha$-Ahlfors regular, i.e.~there exist $\alpha>0$ and $\Cl[c]{c:alpha-lb}, \Cl{c:alpha-ub} \in (0 ,\infty)$ such that the volume growth condition
\begin{equation}
\label{eq:intro_sizeball} \tag{$V_{\alpha}$}
\Cr{c:alpha-lb}r^{\alpha}\leq \lambda(B(x,r))\leq \Cr{c:alpha-ub}r^{\alpha}\quad \text{ for all }x\in{G}\text{ and }r\geq1,
\end{equation}
is satisfied. In particular, \eqref{eq:intro_sizeball} entails that 
the sets $B(x,r)$ have finite cardinality for all $x \in G$ and $r>0$. We furthermore impose that there exist constants $\Cl[c]{c:green-lb}, \Cl{c:green-ub} \in (0,\infty)$ and an exponent $\nu \in (0,\infty)$ such that the Green function $g$ on $G$ (cf.~\eqref{eq:green}) satisfies
\begin{equation}
\label{eq:intro_Green}\tag{$G_{\nu}$}
\begin{split}
& \Cr{c:green-lb}\leq g(x,x)\leq  \Cr{c:green-ub}
\text{ and }  
 \Cr{c:green-lb} d(x,y)^{-\nu}\leq g(x,y)\leq  \Cr{c:green-ub} d(x,y)^{-\nu} \quad \text{ for all }x \neq y\in{G}.
\end{split}
\end{equation}
In particular \eqref{eq:intro_Green} implies that $X$ is transient. The conditions \eqref{eq:ellipticity}, \eqref{eq:intro_sizeball} and \eqref{eq:intro_Green} imply that $0<\nu\leq \alpha-2$ in case $d$ is the graph distance (cf.\ \cite{MR2076770}), and we will from now on always assume that this is the case even for general distances $d$. We further note that as proved in \cite[(2.10)]{DrePreRod2} and \eqref{eq:capBallBd}, these conditions imply that
\begin{equation}
\label{eq:lambdabounded}
\begin{split} 
c\leq \lambda_{x,y}\leq \lambda_x\leq C\text{ for all }x\sim y\in{G}.
\end{split}
\end{equation}

The above setup extends to the metric graph $\tilde{\G}$ associated to $\G$ by which each edge $\{x,y\}$ is replaced by a continuous line segment $I_{\{x,y\}}$ of (Euclidean) length $1/(2\lambda_{x,y})$ glued at its endpoints (which correspond to the vertex set $G$).  We refer to \cite[Section 2]{DrePreRod3} for a comprehensive introduction. In this manner the vertex set $G$ is naturally viewed as a subset of $\tilde{\G}$. In particular $X$ extends to a continuous process, the Brownian motion on $\tilde{\G}$. With a slight abuse of notation, we continue to use $X$, $P_{x}$, $x \in \tilde\G$, and $g(x,y)=g_{\tilde{\G}}(x,y)$, $x,y \in \tilde{\G}$ for the extension of quantities defined above to $\tilde{\G}$. 

In the sequel, we define $K\subset \tilde{\mathcal{G}}$ to be bounded if $K\cap G$ is a bounded (or equivalently, finite) set, and introduce its cardinality as
\begin{equation}\label{eq:vol}
|K|\stackrel{\text{def.}}{=}|K\cap G|,
\end{equation}
which is consistent when $K\subset G (\subset \tilde\G)$. We call a set compact if it is compact for the natural geodesic metric on $\tilde{\G}$ which assigns length one to each edge and interpolates linearly within edges. We further write $d(x,K)=\inf_{y\in{K\cap G}}d(y,x)$ for any $K\subset\tilde{\G}$,  with the convention $d(x,K)=\infty$ if $K\cap G=\emptyset$, and call denote by $\tilde{B}(x,r)$ the union of $B(x,r)$ and the set of points in $\tilde{\G}$ which belong to a segment whose endpoints are both contained in the discrete ball $B(x,r)$. 

The Gaussian free field on $\tilde{\G}$ is the continuous centered Gaussian process $(\varphi_x)_{x\in \tilde\G}$ with law $\P$ and corresponding expectation $ 
\E$ such that
\begin{equation}\label{eq:introGFF}
\E[\varphi_x \varphi_y] = g(x,y), \quad \text{ for all }x,y \in \tilde\G.
\end{equation}
Its restriction to $G$ is the discrete Gaussian free field associated to $\G=(G,\lambda)$. We now introduce the percolation problem we are interested in. To this end, we consider
\begin{equation}
\label{eq:intro0}
0, \text{ an arbitrary point in } \G
\end{equation}
and 
\begin{equation}
\label{eq:introKa}
\begin{split}
&{\mathcal{K}}^a \stackrel{\text{def.}}{=} \text{ the connected component of $0$ in $\{ x \in \widetilde{\mathcal{G}}:  \varphi_x \geq a \}$}
\end{split}
\end{equation}
(with ${\mathcal{K}}^a=\emptyset$ if $\varphi_0 <a$). We also introduce for $r\geq1$
\begin{equation}
\label{eq:defVra}
\begin{split} 
\LV_r^a=\sup\big\{|K|:\,K\text{ is a connected component of }\{ x \in \tilde{B}(r):  \varphi_x \geq a \}\big\}.
\end{split}
\end{equation}
Note that in case $\lambda\equiv1$ (which is for instance the case on $\Z^d$), the set of edges entirely included in $\{x\in{\tilde{\G}}:\,\phi_x\geq a\}$ has exactly the same law as the set of open edges described around \eqref{eq:disrete_model}, see \cite[(1.6)]{DrePreRod5} and above. In particular, $|\K^a|$ and $M_r^a$ have the same law as the corresponding sets introduced above Theorem~\ref{T:critvol-Zd}. The set $\K^a$ from the introduction actually corresponds to the set $\K^a\cap \Z^d$ from \eqref{eq:introKa}, but this slight abuse of notation will be of no consequence since we focus on $|\K^a|$ ($=|\K^a\cap \Z^d|$) in the introduction, and will from now on only consider $\K^a\subset\tilde{\G}$ as defined in \eqref{eq:introKa}.

In this article, we are concerned with the
behavior of the quantities $|\mathcal{K}^a|$ and $\LV_r^a$, which takes on values in $\{0,1,\dots\} \cup \{ \infty\}$, as $a\in \R$ varies. Consider the associated order parameter ${\theta}_0(a) \stackrel{\text{def.}}{=}\P({\mathcal{K}}^a \text{ is bounded})$, $a \in \R.$ It turns out that, under the assumption \eqref{eq:intro_Green}, the value $a=0$ is critical in the sense that
\begin{equation} \label{T1_sign}
{\theta}_0(a) =1 \text{ if and only if }  a \geq 0;
\end{equation}
see \cite[Theorem 1.1 and Lemma 3.4(2)]{DrePreRod3}. In fact \eqref{T1_sign} holds under much weaker assumptions on $\G$. In view of \eqref{T1_sign}, the regime $a> 0$ will be referred to as \emph{subcritical} and the regime $a< 0$ as \emph{supercritical}. 

The behavior of the volume $|\mathcal{K}^a|$ that we aim to investigate is intimately linked to that of the radius of $\mathcal{K}^a$ (see for instance \eqref{eq:qPsiBd} and \eqref{eq:critvol-ub} below). It will be convenient to parametrize the latter in terms of a function $q(\cdot)$, as follows. With hopefully obvious notation, we write $\{U \stackrel{\geq a}{\longleftrightarrow} V\}$, for $U,V \subset \tilde\G$, in the sequel to denote the event that $U$ and $V$ are connected by a (continuous) path in $\{ \varphi \geq a\} = \{x \in \tilde\G: \varphi_x \geq a\}$. For $B \subset \tilde{\G}$ we write $\partial B$ to denote the topological boundary of $B$. We write $q: [1,\infty)\to  (0,\infty]$ for the function defined as
\begin{equation}
    \label{eq:qPsiBd}
q(r)\stackrel{\text{def.}}{=}\sup_{x \in G} r^{\nu/2} \P \big( x \stackrel{\geq 0}{\longleftrightarrow} \partial \widetilde{B}(x,r) \big), \ r \geq 1.
\end{equation}

\begin{Rk}[The function $q$] \label{R:1-arm}
Let us briefly comment on the choice of parametrization in \eqref{eq:qPsiBd}. It turns out to be particularly well-suited in \lq low dimensions,\rq\ in which $q$ is typically of unit order uniformly in $r$ under appropriate assumptions on $\G$, as we now detail. First, on any graph satisfying \eqref{eq:ellipticity} and \eqref{eq:intro_Green} for some $\nu > 0$, by \cite[(1.22) and (1.23)]{DrePreRod5} one knows that
\begin{equation}
\label{eq:boundq}
    c\leq q(r)\leq Cr^{\frac{(\nu-1)\vee 0}2}\log(r)^{\frac{1\{\nu=1\}}2} \ \text{ for all }r\geq1.
\end{equation}
In particular, \eqref{eq:boundq} already yields that $q(r) \leq C$ for all $r \geq 1$ when $\nu <1$. Note that \eqref{eq:boundq} does not even rely on a notion of volume growth or dimension for $\G$, but just on the decay of the Green function.  Under the additional hypothesis that \eqref{eq:intro_sizeball} holds, a sharp upper bound matching (up to constants) the lower bound in \eqref{eq:boundq} was recently shown in a series of works \cite{DrePreRod8,DrePreRod9,cai2024onearm} in `low dimensions,' as we detail below. In summary, these results entail that, for $\G$ satisfying \eqref{eq:ellipticity}, \eqref{eq:intro_Green} and \eqref{eq:intro_sizeball}, 
\begin{equation}\label{eq:boundq-new}
\sup q \equiv \sup_{r \geq 1} q(r)< \infty, \quad \text{ when }\alpha > 2\nu \text{ and on } \Z^\alpha, \alpha=3,4,5
\end{equation}
(with unit weights between nearest neighbors in the case of $\Z^{\alpha}$, for which $\nu=\alpha-2$). For the reader's orientation, the bound \eqref{eq:boundq-new} was established in \cite{DrePreRod8,DrePreRod9} when $\alpha > 2\nu$, which comprises the case of $\Z^\alpha$, $\alpha=3$ with unit weights;  the ref.~\cite{DrePreRod8} also contains a bound in the case $\alpha = 2\nu$ (comprising $\Z^4$) by which $q(r)\leq (\log r)^C$ for all $r\geq1$. 
The case of $\Z^{\alpha}$, $\alpha=3,4,5$ was then obtained by different methods in \cite{cai2024onearm}. It was further shown in \cite{cai2024onearm} that $q(r)= 1 \vee \exp\{ \sqrt{\log r} \log \log r \}$ on $\Z^6$. On the other hand, on square lattices with dimension larger than or equal to $7$, $q(r)$ is of order $r^{\frac{\nu}2-2}$ by \cite{cai2023onearm},  reflecting mean-field behavior.
\end{Rk}

In order to put our results below into context, we also observe that by repeating the proof of \cite[Corollary 1.6]{DrePreRod5}, but substituting \eqref{eq:qPsiBd} for the upper bounds \cite[(1.22)--(1.23)]{DrePreRod5} used therein, one obtains that if $\G$ satisfies \eqref{eq:ellipticity}, \eqref{eq:intro_sizeball} and \eqref{eq:intro_Green}, then 
\begin{equation}
\label{eq:critvol-ub}
       \P(|\mathcal{K}^0| \geq n) \leq  C q\big(n^{\frac2{2\alpha-\nu}}\big)n^{-\frac{\nu}{2\alpha - \nu}}, \text{ for all $n \geq 1$.}
\end{equation}

We are now ready to state the first main result of this article, which concerns the volume $|\K^a|$ of a given cluster. Recall our convention regarding constants from the end of Section~\ref{sec:intro}.

\begin{The} \label{thm:volLB}
Assume $\G$ satisfies \eqref{eq:ellipticity}, \eqref{eq:intro_sizeball} and \eqref{eq:intro_Green} for some $\alpha,\nu >0$, and that $\sup q < \infty$ (cf.~\eqref{eq:qPsiBd} and Remark~\ref{R:1-arm}). There exist $c,C\in (0,\infty)$ such that for any  $a \in [-1,1] $ and $n \ge 1$,
    \begin{equation} \label{eq:tailEstVolgen}
\P\big( n \leq |\mathcal K^a| <\infty \big) \ge
 c n^{-\frac{\nu}{2\alpha - \nu}} \exp \big\{ - C |a|^{2-\frac{\nu}{\alpha}}n^\frac{\nu}{\alpha} \big\}.
\end{equation}
\end{The}

Let us now record an interesting direct consequence of \eqref{eq:tailEstVolgen} for the moments of $|\K^a|$, which follows from the above in combination with the previously established law of the capacity, cf.~\eqref{eq:capTail} below. 

\begin{Cor}
\label{C:offcritvol-mom}
   Assume $\G$ satisfies \eqref{eq:ellipticity}, \eqref{eq:intro_sizeball} and \eqref{eq:intro_Green} for some $\alpha,\nu >0$, then for each $k\geq 1$ 
    there exists $c=c(k) \in (0,\infty)$ such that for all $a \in [-1,1] \setminus \{0\}$, 
    \begin{align} \label{eq:offcritvol-mom1}
    \E\big[|\mathcal K^{a} |^k\cdot 1\{|\mathcal{K}^a|<\infty\}\big]  \ge c |a|^{1-2k},
    \end{align}
    and if in addition, $\sup q < \infty$, then
    \begin{align} \label{eq:offcritvol-mom2}
    \E\big[|\mathcal K^{a} |^k\cdot 1\{|\mathcal{K}^a|<\infty\}\big]  \ge c|a|^{1-k\frac{2\alpha-\nu}{\nu}}.
    \end{align}
\end{Cor}
The proofs of Theorem~\ref{thm:volLB} and Corollary~\ref{C:offcritvol-mom} appear in Section~\ref{sec:den}.
We believe that \eqref{eq:offcritvol-mom1} is optimal up to constants in the mean-field regime $2\alpha< 3\nu$, whereas \eqref{eq:offcritvol-mom2} is optimal up to constants when $2\alpha> 3\nu$.  In the upper-critical dimensions $2\alpha=3\nu$, the two bounds coincide and are presumably correct up to logarithmic corrections. Indeed, if an upper bound matching \eqref{eq:offcritvol-mom2} is satisfied, then the critical exponent $\Delta$, such that $\E\big[|\mathcal K^{a} |^{k+1}\cdot 1\{|\mathcal{K}^a|<\infty\}\big]=a^{-\Delta+o(1)}\E\big[|\mathcal K^{a} |^k\cdot 1\{|\mathcal{K}^a|<\infty\}\big]$ for each $k\geq1$, would be equal to $\frac{2\alpha}{\nu}-1$ outside of the mean-field regime (in particular, it would not depend on $k$), matching the prediction from \cite[Table~1]{DrePreRod5}. Since \eqref{eq:tailEstVolgen} is a simple consequence of \eqref{eq:offcritvol-mom2}, this justifies why we also believe \eqref{eq:tailEstVolgen} to be optimal up to constants when $2\alpha>3\nu$. In the mean-field regime, an upper-bound matching \eqref{eq:offcritvol-mom1} would imply that $\Delta=2$, which would be the same exponent as for Bernoulli percolation in the mean-field regime, see \cite[(10)]{MR0923855}. It is a very interesting open question to prove a matching upper bound to \eqref{eq:offcritvol-mom2} outside of the mean-field regime, or even more generally to \eqref{eq:tailEstVolgen}. We refer to Proposition~\ref{pro:easyoffcriticalbounds} for a partial result in that direction, which shows that the probability in \eqref{eq:tailEstVolgen} decays at most as $\exp\{-Ca^2n^{\frac\nu\alpha}\}$.

We now turn to our results concerning the volume of the largest cluster $\LV_r^a$ in a ball. Similarly as in \eqref{eq:critvol-ub}, we first focus on upper bounds at criticality. As indicated above \cite[Conjecture~A]{werner2020clusters}, one can deduce a bound on the average value of $\LV_r^0$ from \cite[Proposition~5.2]{MR3502602}, which we formally prove in the following proposition.
\begin{Prop}
\label{pro:upperMRa}
Assume $\G$ satisfies \eqref{eq:ellipticity}, \eqref{eq:intro_sizeball} and \eqref{eq:intro_Green} for some $\alpha,\nu >0$. There exists $C\in{(0,\infty)}$ such that for all $r\geq1$, 
\begin{equation}
\label{eq:upperMR0}
\begin{split} 
\E\big[\LV^{0}_r\big]  \le Cr^{\alpha-\frac{\nu}2},
\end{split}
\end{equation}
and thus $\P(M_r^0\leq tr^{\alpha-\nu/2})\geq 1-Ct^{-1}$ for all $t>0$.
\end{Prop}
\begin{proof}
For each $x\in{B(r)}$, let us denote by $\K_r^0(x)$ the cluster of $x$ in $\{y\in{\tilde{B}(r)}:\,\phi_y\geq0\}$. Note that by \eqref{eq:defVra} we have
\begin{equation*}
\begin{split} 
(\LV_r^0)^2=\sup_{x\in{B(r)}}|\K_r^0(x)|^2\leq \sum_{x\in{B(r)}}|\K_r^0(x)|
\end{split}
\end{equation*}
since in the above sum each vertex of $\K_r^0(x)$ is counted $|\K_r^0(x)|$ times for each $x\in{B(r)}$. Hence, combining this with the Cauchy-Schwarz inequality yields
\begin{equation}
\label{eq:upperboundVr0}
\begin{split} 
\E\big[\LV_r^0\big]\leq \E\Big[\sum_{x\in{B(r)}}|\K_r^0(x)|\Big]^{\frac12}\leq Cr^{\frac\alpha2}\sup_{x\in{G}}\E\big[|\K^0_{\tilde{B}(x,2r)}(x)|\big]^{\frac12},
\end{split}
\end{equation}
where we used \eqref{eq:intro_sizeball}, \eqref{eq:lambdabounded} and the inclusion $\tilde{B}(r)\subset\tilde{B}(x,2r)$ for all $x\in{B(r)}$ in the last inequality. Moreover, using \eqref{eq:intro_Green} and \cite[Proposition~5.2]{MR3502602}, which is stated for the loop soup but is also valid for the metric graph Gaussian free field by \cite[Theorem~1]{MR3502602}, we have for all $x\in{G}$ that
\begin{equation}
\label{eq:averagecriticalsize}
\begin{split} 
\E\big[|\K^0_{\tilde{B}(x,2r)}(x)|\big]\leq C\sum_{y\in{B(x,2r)}}d(x,y)^{-\nu}\leq C\sum_{k=0}^{\log_2(2r)}2^{k\alpha-k\nu}\leq Cr^{\alpha-\nu};
\end{split}
\end{equation}
here we used \eqref{eq:intro_sizeball} and \eqref{eq:lambdabounded} in the second inequality after decomposing $B(x,2r)$ in dyadic scales. Combining \eqref{eq:upperboundVr0} and \eqref{eq:averagecriticalsize} yields \eqref{eq:upperMR0}, and hence $\P(M_r^0\geq tr^{\alpha-\nu/2})\leq Ct^{-1}$ by Markov's inequality.
\end{proof}

We now turn to our second main result, which is a lower bound on the size of the largest cluster $\LV_r^a$.

\begin{The}
\label{the:mainMra}
Under the assumptions of Theorem~\ref{thm:volLB}, there exists $c,C\in{(0,\infty)}$ such that for all $a\in{[-1,1]}$ and $r\geq 1$ we have
    \begin{equation} \label{eq:tailEstlargestVolgen}
\P\big( \LV^a_r \geq c(|a|\vee r^{-\frac\nu2})r^{\alpha} \big) \ge
c \exp \big\{ - C (a\vee0)^{2}r^\nu \big\}.
\end{equation}
Moreover, for all  $t,r\geq 1$ with $r^{\alpha-\frac\nu2}\geq Ct$,
    \begin{equation} \label{eq:tailEstlargestVolgena=0}
\P\big( \LV^0_r \geq (1/t)r^{\alpha-\frac\nu2} \big) \ge
 1-\exp(-ct^{C}).
\end{equation}
\end{The}
In view of Proposition~\ref{pro:upperMRa}, the critical bound \eqref{eq:tailEstlargestVolgena=0} is sharp up to constants if $\sup q<\infty$, which is expected to be true below the upper critical dimension, that is if $2\alpha>3\nu$. We refer to Proposition~\ref{pro:meanfieldlower} for a corresponding statement which is expected to be sharp up to logarithms in the mean-field regime $2\alpha\leq 3\nu$.
 
Let us comment on the choice of the scale $(|a|\vee r^{-\frac\nu2})r^{\alpha}$ in the probability in \eqref{eq:tailEstlargestVolgen}. When $a=0$, in view of Proposition~\ref{pro:upperMRa} the largest scale that can be reached by $M_r^a$ is $r^{\alpha-\frac\nu2}$, and \eqref{eq:tailEstlargestVolgen} shows that this scale is indeed reached. When $a<0$, the scale $|a|r^{\alpha}$ simply corresponds to the average number of points in $B(r)$ which belong to the infinite cluster of $\{\phi\geq a\}$ by \cite[Corollary~1.2]{DrePreRod5}. Therefore,  for $a<0$ inequality \eqref{eq:tailEstlargestVolgen} states that with probability $1-\eps$, the largest cluster of $\{\phi\geq a\}$ in $B(r)$ has a size comparable to the size of the infinite cluster in $B(r)$, when this cluster has a large enough chance to intersect $B(r)$, and comparable to the largest critical cluster otherwise. Note that these statements are actually not true on $\Z^d$ for $d\geq7$, see Remark~\ref{rk:final},\ref{rk:boundinanydimension} for details, and hence the assumption $\sup q<\infty$ in Theorem~\ref{the:mainMra} is necessary.

For $a>0$, \eqref{eq:tailEstlargestVolgen} simply gives a lower bound on the probability that the subcritical cluster at level $a$ has the same number of points as the typical supercritical cluster at level $-a$.  Note moreover that if the conjectured upper bound matching \eqref{eq:tailEstVolgen} up to constant is satisfied, one can use a union bound to derive an upper bound matching \eqref{eq:tailEstlargestVolgen} up to constants for $a>0$ and $r\geq Ca^{-2/\nu}\log(a^{-1})^{\frac1\nu}$. Since $\P(\LV_r^a\geq cr^{\alpha-\frac\nu2})$ is of constant order whenever $r< a^{-2/\nu}$ by \eqref{eq:tailEstlargestVolgen} and a trivial upper bound, the usual scaling ansatz indicates that the lower bound \eqref{eq:tailEstlargestVolgen} should be sharp up to constants.

\section{Some preparation}
\label{sec:pre}
Throughout the remainder of this article we tacitly assume that $\G$ is a weighted graph as introduced in Section~\ref{sec:mainresults}, on which the assumptions \eqref{eq:ellipticity}, \eqref{eq:intro_sizeball} and \eqref{eq:intro_Green} (see Section~\ref{sec:mainresults}) are satisfied.
In this section introduce some more notation and gather a few preliminary results within this framework that will be used later, including initial bounds on the decay of the cluster of the origin being of large volume, cf.\ Proposition \ref{pro:easyoffcriticalbounds}. 

\subsection{The graph $\G_K$}
\label{sec:GK}

We start by introducing the weighted graph $\G_K$ that will play an important role, where
\begin{equation}\label{eq:K-ass}
\text{$K\subset\tilde{\G}$ is compact and connected,}
\end{equation}
(see above \eqref{eq:introGFF} for terminology; note that this includes the possibility $K=\emptyset$). Note in particular for later purposes that the cluster $\K^a$, see \eqref{eq:introKa}, satisfies \eqref{eq:K-ass} if $|\K^a|<\infty$.
Informally, $\G_K$ is such that its associated metric graph $\tilde{\G}_K$ corresponds to the unbounded connected component of $\tilde{\G}\setminus K$, with killing on $K$. Contrary to $\G=(G,\lambda)$, the graph $\G_K=(G_K,\lambda^K, \kappa^K)$ will additionally have a non-zero killing measure that we denote by $\kappa^K$. The definition of the metric graph $\tilde{\G}_K$ in the presence of a killing measure is the same as without a killing measure except that, for each $x\in{G}$ with $\kappa_x>0$, we add to the metric graph $\tilde{\G}_K$ a half-open interval starting at $x$ of Euclidean length $1/(2\kappa_x)$. We refer to \cite[Section 2]{DrePreRod3} for details.

The vertex set $G_K$ of $\G_K$ contains the unbounded connected component of $G\cap K^{\mathsf{c}}$, where $  K^{\mathsf{c}}= \tilde{\G} \setminus K$ (cf.~\cite[(6.5)]{DrePreRod5} regarding its uniqueness), but not only. Moreover, for each vertex $x\in{\partial K}$ and $y$ in the unbounded connected component of $G\cap K^{\mathsf{c}}$ such that there is an edge starting in $y$ which contains $x$ in $\tilde{\G}$, we further add to $G_K$ the point $z_{x,y}$ equidistant (rel.~to Euclidean distance) from $x$ and $y$ on that edge. This completes the specification of $G_K$, and we now turn to the conductances $\lambda^K$ and killing measure $ \kappa^K$. For $y,z$  in the unbounded connected component of $G\cap K^{\mathsf{c}}$ such that $\lambda_{y,z}>0$, we set $\lambda^K_{y,z}= \lambda_{y,z}$. If on the other hand $x\in{\partial K}$ and $y$  in the unbounded connected component of $G\cap K^{\mathsf{c}}$ are such that there is an edge starting in $y$ which contains $x$, we set $\lambda_{y,z_{x,y}}^K=\kappa_{z_{x,y}}^K= \frac{1}{\varrho(x,y)}(=\frac{1}{2\varrho(x,z_{x,y})}$), where $\varrho(\cdot,\cdot)$ denotes the Euclidean distance on the cable of $\tilde{\G}$ containing $x$ and $y$ (recall that if $y'$ is the neighbor of $y'$ corresponding to the cable containing $x$ then this cable has total length $\varrho(y,y')= 1/2\lambda_{y,y'}$). The measure $\kappa^K$ vanishes at all other vertices.

 The weights $\lambda^K$ are chosen so that the Euclidean length of the cable between $y$ and $z_{x,y}$ is the same on $\tilde{\G}$ and $\tilde{\G}_K$, the metric graph associated to $\G_K$, and $\kappa^K$ so that the length of the half-open interval starting in $z_{x,y}$ on $\tilde{\G}_K$ is the same as the length of the interval between $z_{x,y}$ and $x$ on $\tilde{\G}$. In particular, one can identify $\tilde{\G}_K$ with the largest connected component of $\tilde{\G}\setminus K$, and identify the diffusion on $\tilde{\G}_K$, whose law will be denoted by $\PPK_{\cdot}$, with the diffusion on this component, killed when hitting $K$. The Green function $g_{\tilde{\G}_K}$ on $\tilde{\G}_K$ (see above \eqref{eq:introGFF} for notation) can thus be identified with the restriction to $\tilde{\G}_K$ of the Green function on $\tilde{\G}$ killed when hitting $K$. In the same vein, one also identifies the law $\PK$ of the Gaussian free field on $\tilde{\G}_K$, defined by \eqref{eq:introGFF} with $g_{\tilde{\G}_K}$ in place of $g=g_{\tilde{\G}}$, with the law of the Gaussian free field killed on $K$. Note that all of the above is well-defined and consistent if $K$ in \eqref{eq:K-ass} is the empty set, in which case $\G_{\emptyset}=\G$, $\PPKK{\emptyset}_{\cdot}=P_{\cdot}$ and $\PKK{\emptyset}=\P$.

\subsection{Markov property}

Let us now recall the strong Markov property for the Gaussian free field, see e.g.~\cite[(1.19)]{MR3492939} for details. For  $O\subset\tilde{\G}$ let $\mathcal{A}_O$ denote the $\sigma$-algebra $\sigma({\phi}_x,\,x\in{O})$. For compact $K\subset\tilde{\G}$ we consider $\mathcal{A}_K^+=\bigcap_{\eps>0}\A_{K^\eps}$, where $K^\eps$ is the open $\eps$-ball around $K$ for the geodesic distance on $\tilde{\G}$ which assigns length one to each edge. We say that a (random) connected set $\K \subset \tilde \G$ measurable with respect to $\phi$ is \textit{compatible} if it is always compact and if $\{\K\subset O\}\in{\A_O}$ for any open set $O\subset\tilde{\G}.$ For such $\K$ let
\begin{equation}
\label{eq:Markov1}
    \begin{split}
    \mathcal{A}_{\K}^+\stackrel{\text{def.}}{=}\Big\{A\in{\A_{\tilde{\G}}}:\,A\cap\{\K\subset K'\}\in{\A_{K'}^+}\ &\text{for all compact }K'\subset\tilde{\G}  \text{ with } \mathring{K}' \neq \emptyset\Big\}.
        \end{split}
\end{equation}
For any $K$ as in \eqref{eq:K-ass} and measurable set $\K\subset\tilde{\G}_K$ with respect to $\phi$ under $\PK$, with a slight abuse of notation we will also write $\mathcal{A}_{\K}^+$ for the $\sigma$-algebra \eqref{eq:Markov1} but for the graph ${\G}_K$ (with associated cable system $\tilde\G_K$), and the dependency on $K$ will always be clear from context. 
The strong Markov property on the graph $\G_K$ then asserts that for any compact and connected set $K\subset\tilde{\G}$ (cf.~\eqref{eq:K-ass}) and any random compatible and connected set $\K \subset \tilde{\G}_K,$ 
\begin{equation}
\label{eq:Markov2}
(\phi_x-h_{\K}^{K}(x))_{x\in{\tilde{\G}_{\K\cup K}}}\text{ under }\PK(\, \cdot\,|\,\mathcal{A}_{\K}^+)\text{ has law }\PKK{K\cup\K};
\end{equation}
here, for a compact set $K'\subset\tilde{\G}_K$, we write $h_{K'}^{K}(x)=\PPK_x(H_{K'}<\infty)=P_x(H_{K'}<H_K)$ for the entrance probability of $K'$ on the graph $\G_K$, with $H_{K'} =  \inf\{t \geq 0 : X_t \in K' \}$.

\subsection{Capacity as a tool for the Gaussian free field}

We now collect some important results about the capacity of sets, in particular of level set clusters of the Gaussian free field in a box. For a compact and connected set $K\subset \tilde{\mathcal{G}}$, we write $e_{K}= e_{K,\tilde{\G}}$ for the equilibrium measure of $K$ relative to $\tilde{\G}$, which is supported on a finite set included in $\partial K$ (see for instance \cite[(2.16)]{DrePreRod3} for its definition in the present context). Its total mass
 \begin{equation}
 \label{eq:cap}
 \text{cap}(K)= \text{cap}_{\tilde{\G}}(K) \stackrel{\text{def.}}{=} \int {\rm d}e_{K,\tilde{\G}} \ (< \infty)
 \end{equation}
 is the capacity of $K$. In view of \cite[(5.7)]{DrePreRod5}, the capacity of a ball satisfies
\begin{equation}
\label{eq:capBallBd}
\begin{split} 
cr^\nu \leq {\rm cap}_{\tilde{\G}_K} (B(x,r)) \leq Cr^\nu\text{ for all }x\in{G}.
\end{split}
\end{equation}
for any $K\subset \tilde{B}(0,r)$ and $x\in{G}$ with $d(x,0)\geq CR1\{K\neq\emptyset\}$, where $\tilde{\G}_K$($=\tilde{\G}$ if $K=\emptyset$) is the graph defined in Section~\ref{sec:GK}. One of the main interest of the notion of capacity is that it can be used to estimate hitting probabilities of random walks. Indeed, using a reasoning similar to \cite[(2.17)]{prevost2023passage} (but on the graph $\G_K$), combined with \cite[Lemma~3.1]{DrePreRod2}, there exists $\Cl{C:exit}\in{(1,\infty)}$ such that for all $K,A\subset\tilde{\G}$ as in \eqref{eq:K-ass}, $r\geq1$, $x,y\in{{G}_K}$, $A\subset \tilde{B}(x,r)$, $y\in{B(x,r)}$ with $d(x,K)\geq 2\Cr{C:exit}r$ we have
 \begin{equation}
\label{eq:hittingvscap}
\begin{split} 
 P_y^K\big(H_{A}<H_{\tilde{B}(x,\Cr{C:exit}r)^{\mathsf c}}\big)\geq cr^{-\nu}\mathrm{cap}_{\tilde{\G}_K}(A).
\end{split}
\end{equation}

We now recall the following useful change of measure formula for the Gaussian free field, which was first used in \cite{BoDeZe-95}. A proof of the version we use here can be either found  in the proof of \cite[Lemma~3.2]{GRS21} in the discrete setting, see in particular (3.7) and the last equation therein, or in \cite[(6.16),(6.17)]{DrePreRod5} on metric graphs. For all $K\subset \tilde{\G}$ as in \eqref{eq:K-ass}, $A\subset\tilde{\G}_K$ compact, $b>0$ and events $E\in{\mathcal{B}(\R^{\tilde{\G}_K})}$, writing $\P_{b,A}^K$ for the law of $\big(\phi_x+bP_x^K(H_{A}<\infty)\big)_{x\in{\tilde{\G}_K}}$ under $\PK$, we have that
\begin{align} \label{eq:entropy}
\begin{split}
    &\PK\big({E}\big)
    \ge\PK_{b,A}(E)
    \cdot \exp \bigg\{ - \frac{b^2 {\rm cap}_{\tilde{\G}_K}(A) + 2/{\rm e}}{2\PK_{b,A}(E)} \bigg\}.
\end{split}
\end{align}

We now collect some properties concerning the capacity of clusters. We write $\mathcal{K}^a(x)$, $x \in G$, to denote the cluster of $x$ in $\{\varphi \geq a\}$, defined as in \eqref{eq:introKa} with $x$ in place of $0$, whence $\mathcal{K}^a= \mathcal{K}^a(0)$. It  follows from \cite[Corollary~3.3,(1), Lemma~3,4,(2) and Theorem~3.7]{DrePreRod3} that on any graph $\G$ such that $\sup_{x\in{G}}g_{\tilde{\G}}(x,x)<\infty$, the density of $\mathrm{{\rm cap}}( \K^{a}(x))$ under $\P$ is given by \cite[(3.8)]{DrePreRod3}. Since for all $x\in{K}$, $g_{\tilde{\G}_K}(x,x)\leq g_{\tilde{\G}}(x,x)\leq C$ by \eqref{eq:intro_Green}, this result can also be applied to the graph $\G_K$.  In particular one deduces that for all sets $K\subset\tilde{\G}$ as in \eqref{eq:K-ass} (possibly with $K=\emptyset$), $a\in{\R}$, $x\in{G}$ and $s\geq C$, 
\begin{equation}
\label{eq:capTail}
    cs^{-\frac12}e^{-Ca^2s} \leq \PK\big(s \leq \mathrm{cap}_{\tilde{\G}_K}( \K^{a}(x)) < \infty 
    \big)
    \leq C \big({g_{\tilde{\G}_K}(x,x)}s\big)^{-\frac12} e^{-ca^2s}.
\end{equation}
The tail estimate \eqref{eq:capTail} will be useful in various places. In particular, close to criticality, combining it with \eqref{eq:qPsiBd}, one can obtain the following local version of \eqref{eq:capTail}. Let us first introduce, for a closed set $\tilde{B} \subset \tilde\G_K$, $x\in{G_K}$ and $a \in \mathbb{R}$, the cluster of $x$ in $\{\varphi \geq a \}\cap \tilde{B}$, defined as
\begin{equation}
\label{eq:Kr}
{\K}^{a}_{\tilde B}(x) =\big\{ y \in  \widetilde B: y \leftrightarrow x \text{ in } \{\varphi \geq a \}\cap \tilde{B} \big\}.
\end{equation}
Note that  ${\K}^{a}_{\tilde \G_K}(0)= {\K}^{a}$, the cluster of $0$ in $\{\varphi \geq 0\}$ under $\PK$, and that ${\K}^{a}_{\tilde B}(x)$ is empty unless $x \in G_K\cap \tilde B$. In the sequel, we will typically be interested in the case that $\widetilde{B}=\widetilde{B}(x,r)$ for some $x \in G$ and $r \geq 1$. When $\tilde{B}=\tilde{B}(r)$(=$\tilde{B}(0,r)$), we abbreviate $\K_r^a(x)$ for $\K_{\tilde{B}(r)}^a(x)$.

\begin{Lemme}
\label{lem:capbox}
There exists $\Cl[c]{C:lawcaploc}>0$ such that 
the following holds true: For all $x \in \G$ and for all $K$ as in \eqref{eq:K-ass} with $\tilde{B} \cap K = \emptyset$ (abbreviating $\tilde B=\tilde{B}(x,r)$), for all $s,r\geq1$, $a\geq0$ and $x\in{G}$ such that,  $s\leq \Cr{C:lawcaploc}r^{\nu}q(r)^{-2}$, as well as if $s\leq 1\vee a^{-2}$ in case $a\neq0$, we have

\begin{equation}\label{eq:clustCapLocal}   
\begin{split}
\PK\Big(\mathrm{cap}_{\tilde{\G}_K}\big( \K^{a}_{\tilde{B}}(x)\big)\geq s \Big)
\ge cs^{-1/2}\exp\{-Ca^2\}.
\end{split}
\end{equation}
\end{Lemme}

\begin{proof}
We start with the case $a\in{[0,1]}$. Combining \eqref{eq:capTail} for $s\leq a^{-2}$ and \eqref{eq:qPsiBd}, we have that
\begin{equation}  
\label{eq:capinball}
\begin{split}
\PK\big(\mathrm{cap}_{\tilde{\G}_K}( \K^{a}_{\tilde{B}}(x))\geq s \big)
&\geq\PK\big( \mathrm{cap}_{\tilde{\G}_{K}}\big({\K}^a(x)\big)\geq s  \big)-\PK\big({\K}^a(x)\not\subset \tilde{B}\big) \geq cs^{-\frac12}-r^{-\frac\nu 2}q(r).
\end{split}
\end{equation}
Note that in the last inequality above, we used the fact that $\K^a(x)\subset \K^0(x)$ as well as the definition of $q$ in \eqref{eq:qPsiBd}. While this definition is for the field $\phi$ on the graph $\tilde{\G}$ instead of on $\tilde{\G}_K$, we will now argue that the resulting upper bound on $\P\big({\K}^a(x)\not\subset \tilde{B}\big) \leq r^{-\frac\nu 2}q(r) $ remains true on the graph $\tilde{\G}_K$ (i.e.,~with $\P^K$ in place of $\P$), as used above. Indeed, using the isomorphism with loop soups \cite[Proposition~2.1]{MR3502602}, combined with the restriction property of loop soups \cite[Section~6]{MR3238780}, one can infer that $\K^0(x) $  has the same law under $\PK$ as the cluster of $x$ in a loop soup on $\tilde{\G}$ (up to the choice of an independent sign, which is inconsequential), restricted to the loops which avoid $K$. This cluster of loops avoiding $K$ is included in the entire cluster of $x$ consisting of all loops, which has the same law (up to a sign)  as $\K^0(x)$ under $\P$ by the isomorphism again. In particular, it follows that the law of $\K^0(x)$ under $\PK$ is stochastically dominated by the law of $\K^0(x)$ under $\P$. This finishes the proof of \eqref{eq:clustCapLocal} when $a\in{[0,1]}$ upon taking $\Cr{C:lawcaploc}$ small enough. 

When $a\geq1$, we necessarily have  $s\leq 1$  under our assumptions. Moreover, it follows from \eqref{eq:capBallBd} that $\mathrm{cap}_{\tilde{\G}_K}(B(x,C))\geq \mathrm{cap}(B(x,C))\geq 1$ if $C$ is a large enough constant, and hence for  all $s \leq 1$, 
\begin{multline*}
\PK\big(\mathrm{cap}_{\tilde{\G}_K}( \K^{a}_{\tilde{B}}(x))\geq s \big) \geq \black 
\PK\big(\mathrm{cap}_{\tilde{\G}_K}( \K^{a}_{\tilde{B}}(x))\geq 1 \big)
\\\geq \PK\big(\phi_y\geq a\text{ for all }y\in{ \widetilde{B}(x,C)}\big) \geq c\exp\{-Ca^2\},
\end{multline*}
where the last inequality follows by combining classical bounds on the tail of Gaussian variables, \eqref{eq:intro_Green}, the FKG-inequality and \cite[(1.6)]{DrePreRod5} (for instance, forcing the field to exceed value $2a$ on all vertices of $B(x,C)$ in order to obtain the desired Gaussian tail in $a$ when applying \cite[(1.6)]{DrePreRod5}). 
\end{proof}

The next result presents useful comparison estimates between capacity and volume of a generic set $K$. Recall that $\G$ satisfies \eqref{eq:ellipticity}, \eqref{eq:intro_sizeball}, \eqref{eq:intro_Green}, cf.~the beginning of this section. 

\begin{Lemme}\label{L:cap-LB} 
For all $K \subset \tilde{\G}$  as in \eqref{eq:K-ass}, we have
\begin{align}
  &\label{eq:subadd} \textnormal{cap}(K) \le C \cdot |K|,  \text{ if $K\cap G\neq\emptyset$,  and} \\
  &\label{eq:cap-LB} \textnormal{cap}(K) \geq c \cdot|K|^{\frac\nu\alpha}.
\end{align}
\end{Lemme}
The inequality in \eqref{eq:subadd} expresses the
sub-additivity inherent to $\text{cap}(\cdot)$ within the present framework, while the second bound is essentially saturated when $K$ is a ball.

\begin{proof}
We first show \eqref{eq:subadd}, and begin by reducing to the case where $\partial K \subset G$. Indeed, for a given $K \subset \tilde{\G}$ with $|K|< \infty$, consider $K' \supset K$ defined as comprising all the (closed) cables intersected by $K$.
Clearly $K \subset K'$ and $\partial K' \subset G$ hold by construction, and, as we now explain,
\begin{equation}
\label{eq:KvsK'} |K'| \leq C |K|
\end{equation}
for some constant $C \in (0,\infty)$ uniform in $K$.
Indeed, by assumption on $K$, if the closed cable $I_{e}$ for some $e=\{x,y\}$ (with $\lambda_{x,y}>0$) is a subset of $K'$, then $x$ or $y$ must necessarily be contained in $K$ since $K$ is connected and intersects $G$; thus, \eqref{eq:KvsK'} follows since the maximal degree of $\G$ is finite owing to \eqref{eq:ellipticity}. Now using that, by definition of the equilibrium measure of $K$, its support is a subset of $\partial K$, combined with \eqref{eq:lambdabounded}, one readily obtains that 
$$
 \textnormal{cap}(K) \stackrel{K \subset K'}\leq \textnormal{cap}(K') \leq (\sup_{x \in G} \lambda_x)|{\partial}K'| \stackrel{\partial K' \subset G}{\leq} C|K'| \stackrel{\eqref{eq:KvsK'}}{\leq}  C'|K|.
$$
We now turn to \eqref{eq:cap-LB}. In fact we prove the stronger statement that 
\begin{equation}\label{eq:cap-LB-pf}
\textnormal{cap}(K \cap G) \geq c \cdot|K|^{\frac\nu\alpha} \, \big( = c \cdot|K \cap G|^{\frac\nu\alpha} \big),
\end{equation}
from which the desired bound is immediate by monotonicity of $\text{cap}(\cdot)$. Since \eqref{eq:cap-LB-pf} deals effectively with $(K \cap G) \subset G$, we assume without loss of generality in the sequel that $K \subset G$ (a finite set). Let 
\begin{equation}
    \label{eq:cap-LB-pf1}
v_K(x)\stackrel{\text{def.}}{=} \sum_{y \in K} g(x,y).
\end{equation}
We aim to show that under the above assumptions on $\G$,
\begin{equation}
    \label{eq:cap-LB-pf2}
v_K \stackrel{\text{def.}}{=} \sup_{x \in K} v_K(x) \leq C |K|^{\frac{\alpha-\nu}{\alpha}}. 
\end{equation}
If \eqref{eq:cap-LB-pf2} holds, then using a well-known variational characterization of the capacity (see for instance \cite[(2.21)]{DrePreRod3} and pick $\mu$ the uniform distribution on $K$), by which $\text{cap}(K) \geq  |K|/v_K$, the claim \eqref{eq:cap-LB} immediately follows. Now in order to establish \eqref{eq:cap-LB-pf2}, we observe that in view of \eqref{eq:cap-LB-pf1} and applying \eqref{eq:intro_Green}, we get that
\begin{equation}
    \label{eq:cap-LB-pf3}
v_K(x) \leq C\sum_{k \geq1} \alpha_k 2^{-k\nu}, \text{ for all } x\in K,
\end{equation}
where $\alpha_1= |K \cap B(x, 2)|$ and $\alpha_k= |K \cap (B(x, 2^k) \setminus B(x, 2^{k-1})|$ for $k \geq 2$. 
Note that $\sum_{k \geq 1} \alpha_k = |K|$, so these sums are effectively over finitely many summands only. 
 Observe now that for $\beta=(\beta_k)_{k \geq0}$ a sequence of non-negative integers, the map $\beta \mapsto \sum_{k \geq1} \beta_k 2^{-k\nu} $ is non-decreasing, where we define $\beta\leq \beta'$ if $\sum\beta_k\leq \sum\beta'_k$ and $\beta_k\leq \beta'_k$ for all $k$ such that $\beta'_k\neq 0$. One thus obtains an upper bound on the right-hand side of \eqref{eq:cap-LB-pf3} by choosing $\bar\alpha_1=|B(x,2)|$, $\bar\alpha_k=| B(x, 2^k) \setminus B(x, 2^{k-1})|$, $k \geq 2$, which satisfy   $ \bar\alpha_k \leq C 2^{\alpha k}$ using \eqref{eq:intro_sizeball} and \eqref{eq:lambdabounded}, and summing over $k  \leq  k_0 \stackrel{\text{def.}}{=} \frac1{\alpha}\log _2 |K| + C'$, so as to ensure that $ (\sum_{k \geq 1} \alpha_k =) \black |K| \leq \sum_{ 1 \leq k  \leq k_0}\bar \alpha_k$ and $\alpha_k\leq\bar\alpha_k$ for all $k\leq k_0$. Overall, recalling that $\alpha\geq \nu+2$ as assumed below \eqref{eq:intro_Green},  this yields that
$$
v_K(x) \leq C  \sum_{1 \leq k  \leq k_0}\bar \alpha_k 2^{-k\nu}  \leq C' \sum_{1 \leq k \leq k_0} 2^{k(\alpha-\nu)} \leq C'' 2^{k_0(\alpha-\nu)}= C |K|^{\frac{\alpha-\nu}{\alpha}},
$$
for all $x \in K$, whence \eqref{eq:cap-LB-pf2}, which completes the proof.
\end{proof}

\begin{Rk}[Volume functionals on $\tilde\G$] \label{R:vol-forms}
    The condition on $K$ in \eqref{eq:subadd} to have non-empty intersection with $G$ is convenient. Note that some assumption along these lines is needed since it may well be that $|K|=0$ while $K \neq \emptyset$ (in which case $\text{cap}(K) >0$). The (simple) choice \eqref{eq:vol} is fully sufficient for the purposes of the present article, owing to our assumptions on $\G$ (implying for instance that $\lambda$ has bounded weights; see \cite[Lemma A.2]{DrePreRod3}). Arguably, the most natural choice (which we could have equally worked with) is to endow the metric graph $\tilde{\G}$ with a Lebesgue measure $\text{Leb}$, assigning length $1/2{\lambda_{x,y}}$ to the cable between $x$ and $y$, and to define $\text{vol}(K)=\int_{\tilde\G} 1_K \, \mathrm{d}\text{Leb}$ for a closed set $K \subset \tilde\G$ as the corresponding Lebesgue volume.
\end{Rk}

\subsection{Random interlacements and isomorphism}

We now recall the definition of random interlacements as well some useful results, which will be essential in Section~\ref{sec:volneg} owing to a certain isomorphism with the Gaussian free field \cite{MR2892408,MR3502602,MR3492939}, which we recall below. We will in fact be mostly interested in random interlacements on $\tilde\G_K$ rather than just $\tilde\G$, with $K \subset \tilde \G$ as in \eqref{eq:K-ass} (though $K=\emptyset$ is always admissible, in which case $\tilde \G_K= \tilde \G$). 
For such $K$, we define under 
the probability measure $\OPK$ the random interlacement process $(\omega_u)_{u>0}$, which for each $u$ is a point process of doubly infinite trajectories modulo time-shift, and by $\I^u$ the associated interlacement set. We refer to \cite{MR2680403} for a detailed introduction of this process on discrete graphs, as well as to \cite{MR3502602} and \cite[Section~2.5]{DrePreRod3} for its extension to the metric graph $\tilde{\G}_K$ (the second reference in particular handles the case of non-vanishing killing measure, as present when $K \neq \emptyset$). The restriction of the forward parts of the trajectories in $\omega_u$ to a compact and connected set $A\subset \tilde{\G}_K$ (where by forward we refer to the trajectories from their first visit to $A$ onwards) have the following law: they form a Poissonian number with parameter $u\mathrm{cap}_{\tilde{\G}_K}(A)$, of i.i.d.~diffusions,  each with law $${\mathrm{cap}_{\tilde{\G}_K}(A)}^{-1}\sum_{x\in{A}}e_{A,{\tilde{\G}_K}}(x)\PPK_x,$$ where we recall that $\PPK_x$ is the canonical law of the diffusion on $\tilde \G_K$ starting from $x$. The restriction of the random interlacement set $\I^u$ to $A$ is the set of points visited by $\omega_u$ in $A$, and by definition it satisfies  
 \begin{equation}
 \label{eq:defIu}
\begin{split} 
\OPK(\I^u\cap A=\emptyset)=\exp\big\{-u\mathrm{cap}_{\tilde{\G}_K}(A)\big\}.
\end{split}
\end{equation}
 The main interest of random interlacements for the purposes of the present article is the isomorphism (Isom) from \cite{DrePreRod2}, a consequence of which we now recall. By \cite[Lemma~3.4 and Theorem~1.1]{DrePreRod2}, this isomorphism is satisfied on any graph $\G$ such that $\sup_{x\in{G}}g_{\tilde{\G}}(x,x)<\infty$, and in particular for any $K$ as in \eqref{eq:K-ass}, the isomorphism holds true for the graph $\G_K$ if $\G$ satisfies \eqref{eq:intro_Green}. Under the product measure $\OPK\otimes \PK$, we denote by $\mathcal{C}_u$ the closure of the union of the clusters of $\{x\in{\tilde{\G}_K}:|\phi_x|>0\}$ which intersect $\I^u$ (an open set). For simplicity we do not write explicitly the dependence of $\mathcal{C}_u$ on $K$ as it will be clear from the probability measure in question. A direct consequence of the isomorphism (Isom) from \cite{DrePreRod2} is that for any $u>0$,
\begin{equation}
\label{eq:isom}
\begin{gathered} 
     \text{ under }\OPK\otimes \PK,\ \big(\mathcal{C}_u,\big(\phi_x-\sqrt{2u}\big)_{x\in\tilde{\G}_K}\big)\text{ is stochastically }
     \\\text{dominated by }\big(\big\{x\in{\tilde{\G}_{ K}}:\,\phi_x\geq-\sqrt{2u}\big\},\big(\phi_x\big)_{x\in\tilde{\G}_K}\big);
\end{gathered}
\end{equation}
here, for $A,B\subset \tilde{\G}_K$ and $f,g\in{\R^{\tilde{\G}_K}}$ we introduce the partial order $(A,f)\leq (B,g)$ if and only if $A\subset B$ and $f(x)\leq g(x)$ for all $x\in{\tilde{\G}_K}$.  
For brevity, we will often simply refer to \eqref{eq:isom} as the \emph{isomorphism} itself in the following.

Next, let us recall a result from \cite{DrePreRod2} on the so-called local uniqueness of the interlacement set. The local uniqueness event for interlacements that we will work with is defined as
\begin{equation}\label{eq:lu-def}
    \text{LocUniq}_{u,r,t}(x)\stackrel{\textnormal{def.}}{=}\bigcap_{y,z\in{\mathcal{I}^u}\cap \tilde{B}(x,r)}\left\{y\leftrightarrow z\text{ in }\mathcal{I}^u\cap \tilde{B}(x,tr)\right\},
\end{equation}
where the event $\{y\leftrightarrow z\text{ in } K\}$ refers to the existence of a continuous path with range contained in $K$ intersecting both $y$ and $z$. The probabilities of local uniqueness events such as \eqref{eq:lu-def} for interlacements have first been investigated on $\Z^d$ in \cite[Proposition~1]{MR2819660}, and subsequently been extended to any graph $\G$ satisfying \eqref{eq:ellipticity}, \eqref{eq:intro_sizeball} and \eqref{eq:intro_Green} in \cite[Proposition~4.1]{DrePreRod2}. The proof of \cite[Proposition~4.1]{DrePreRod2} can easily be transferred to the graphs $\G_K$ for $K\subset\tilde{\G}$ as in \eqref{eq:K-ass}, as long as $d(x,K)$ is large enough. For our purposes it will be enough to know that, for suitable choice of $\Cl{C:localuniq}$, abbreviating $\textnormal{LocUniq}_{u,r}(x)=\textnormal{LocUniq}_{u,r,\Cr{C:localuniq}}(x)$, the following holds. For all $\varepsilon\in (0,1)$, $u>0$ and $r\geq 1$, and for all $x\in{B(0,2\Cr{C:localuniq}r)^{\mathsf c}}$  and $K\subset \tilde{B}(0,r)$ compact,
   \begin{equation}
    \label{eq:boundlocaluniq}
    \begin{split}
\OPK\big(\textnormal{LocUniq}_{u,r}(x)\big)\geq 1-\eps, \quad \text{ if $ur^{\Cl[c]{C:LBLuniq2}}\geq C(\varepsilon)$};
    \end{split}
    \end{equation}
 see \cite[(5.4) and (5.20)]{DrePreRod5}, which imply \eqref{eq:boundlocaluniq}. We refer to \cite[Theorem 5.1]{DrePreRod5} and \cite[Theorem 1.1]{prevost2023passage} for much stronger results, notably in the regime  $\alpha> 2\nu$. Another useful result concerns the capacity of the interlacement set: for $\eps,u,r,x$ and $K$ as above
 \begin{equation} \label{eq:Iu-capbound}
    \OPK(\text{cap}_{\tilde{\G}_K}(\mathcal{I}^{u} \cap B(x,r)) \geq cr^{\nu}) \geq 1-\varepsilon, \quad \text{if } ur^{\Cr{c:capline}}\geq C(\varepsilon).
\end{equation}
 In order to show \eqref{eq:Iu-capbound}, one uses the following bound, valid for all $s,t,u>0$ and $r \geq 1$, by which, abbreviating $B=B(x,r)$, 
\begin{equation*}
\begin{split}
    \OPK\big(\mathrm{cap}_{\tilde{\G}_K}(\I^u\cap B)\leq tr^{\nu}\big)&\leq  \OEK\big[e^{-s\mathrm{cap}_{\tilde{\G}_K}(\I^u\cap B)}\big]e^{str^{\nu}} \stackrel{\eqref{eq:defIu}}{=} \OEK\big[e^{-u\mathrm{cap}_{\tilde{\G}_K}(\I^s\cap B)}\big]e^{str^{\nu}} \\
    &\leq \big(\OPK(\I^s\cap B(x,r/2) \neq \emptyset) + e^{-u cr^{\Cr{c:capline}}}\big)e^{str^{\nu}}
    \leq \big(e^{-csr^{\nu}} + e^{-u cr^{\Cr{c:capline}}}\big)e^{str^{\nu}}
\end{split}
\end{equation*}
(e.g.~with $\Cl[c]{c:capline}=\nu \wedge \tfrac12$; cf.~\cite[Lemma 3.2]{DrePreRod2}, \eqref{eq:capBallBd}, and since $\mathrm{cap}_{\tilde{\G}_K}(\cdot)\geq \mathrm{cap}(\cdot)$). Taking $s=ur^{\Cr{c:capline}-\nu}$ and $t=c$ small enough, one deduces \eqref{eq:Iu-capbound}. Furthermore, using \cite[Theorem 5.1]{DrePreRod5}, as well as a more careful argument to estimate $\mathrm{cap}(\I^s\cap B)$ using \cite[Lemma~5.3]{DrePreRod5}, one has that
 \begin{equation}\label{eq:c4explicit}
\text{if $\alpha > 2\nu$, then \eqref{eq:boundlocaluniq} and \eqref{eq:Iu-capbound} hold with the choice $\Cr{C:LBLuniq2}=\Cr{c:capline}=\nu$.}
 \end{equation}

\subsection{First off-critical volume bounds}

As we now explain, combining various results from this section, namely, the law of the tail of the capacity \eqref{eq:capTail} for $K=\emptyset$, the general bound on the capacity supplied by Lemma~\ref{L:cap-LB}, the change of measure formula \eqref{eq:entropy}, the isomorphism \eqref{eq:isom}, and the local uniqueness for random interlacements \eqref{eq:boundlocaluniq}, we can already directly obtain the following bounds on the tail of $|\K^a|$, which are sharp up to constants in the exponential for large values of $a$.

\begin{Prop}
\label{pro:easyoffcriticalbounds}
For all $a \in \mathbb{R}$ and $n \geq 1$, one has that
    \begin{equation}
    \label{eq:easyoffcriticalbounds}
 \exp\big\{  - C(a^2\vee1) n^{\frac{\nu}{\alpha}} \big\}\leq  \P(n \leq |\mathcal{K}^a| < \infty) \leq \exp\big\{ -ca^2n^{\frac{\nu}{\alpha}} \big\}.
  \end{equation}
\end{Prop}

\begin{proof}
We start with the proof of the upper bound. Combining \eqref{eq:cap-LB} and the fact that $\text{cap}(K)< \infty$ whenever $|K|< \infty$, which is a plain consequence of the definition of $\text{cap}(\cdot)$, it follows that the event $\{ n \leq |\mathcal{K}^a| < \infty\}$ is contained in $\{ cn^{\nu/\alpha} \leq \mathrm{cap}(\mathcal{K}^a) < \infty\}$. Hence, the second inequality in  \eqref{eq:easyoffcriticalbounds} now immediately follows from the upper bound on the tail of the capacity provided by \eqref{eq:capTail} for $K=\emptyset$ in combination with our assumption \eqref{eq:intro_Green}. 

Let us now turn to the lower bound. Due to the discussion around \eqref{eq:isom}, one knows that the isomorphism (Isom) in the terminology of \cite[Section 1]{DrePreRod3} holds under our assumptions on $\G$, or equivalently, that (Isom') holds, see \cite[(3.14)]{DrePreRod3}.  As a consequence, \cite[Lemma 4.3]{DrePreRod3} is in force and implies that for any $a > 0$ the random variables  $ |\K^{-a}| \cdot 1\{ | \K^{-a}| < \infty\}$ and $ |\K^a| \cdot 1\{ | \K^{a}| < \infty\} $ have the same law under $\P$. Throughout the remainder of the proof we  will therefore assume without loss of generality that 
\begin{equation}\label{eq:apos}
a \geq 0.
\end{equation}
Since the truncation $|\K^{a}| < \infty$ is obsolete for such $a$ (the event holds $\P$-a.s.), our task is reduced to proving that $ \P( |\mathcal K^a| \geq n)$ is bounded from below by the first term in \eqref{eq:easyoffcriticalbounds}. 

Letting $N_r=|\I^{1/2}\cap B(r)|,$ one has that $\OE[N_r]\geq cr^{\alpha}$ by \eqref{eq:intro_sizeball}, \eqref{eq:lambdabounded} and  \eqref{eq:defIu} since $\mathrm{cap}(\{x\})=1/g(x,x)\geq c$ by \eqref{eq:intro_Green}. Moreover, $\OE[N_r^2]\leq Cr^{2\alpha}$ by \eqref{eq:intro_sizeball} and \eqref{eq:lambdabounded}, and it thus follows from the Paley--Zygmund inequality that $\OP(N_r\geq cr^{\alpha})\geq c'$. Combining this with the FKG inequality and \eqref{eq:defIu}, we deduce that $\OP(N_r\geq cr^{\alpha},0\in{\I^{1/2}})\geq c'$. On the intersection of the event $\{N_r\geq cr^{\alpha},0\in{\I^{1/2}}\}$ and the local uniqueness event $\textnormal{LocUniq}_{1/2,r}(x)$, which occurs with probability at least $1-c'/2$ if $r\geq C$ by \eqref{eq:boundlocaluniq}, the connected component of $0$ in  $\I^{ 1/2}\cap B(\Cr{C:localuniq}r)$ contains at least $cr^{\alpha}$ points. Recall the definition of the cluster $\K^{-1}_r$ of $0$ in $\tilde{B}(r)$ from below \eqref{eq:Kr}, and of $\mathcal{C}_{1/2}$ from above \eqref{eq:isom}. Since $\I^{ 1/2 }\subset\mathcal{C}_{ 1/2 }$ by definition, it follows from the isomorphism  \eqref{eq:isom} with $u=1/2$ that for all $r\geq C$,
\begin{equation}
\label{eq:simpleK-1}
\begin{split} 
\P(|\K^{-1}_{\Cr{C:localuniq}r}|\geq \Cl[c]{c:choicen2}r^{\alpha})\geq c,
\end{split}
\end{equation}
for suitable constants $\Cr{c:choicen2},c\in{(0,\infty)}$. Applying the entropy bound \eqref{eq:entropy} with $K=\emptyset$, $b=a+1$, $A=B({\Cr{C:localuniq}r})$ and $E=\{|\K^a_{\Cr{C:localuniq}r}|\geq n\}$ we thus deduce that for all  for all $a$ as in \eqref{eq:apos}  and $n\geq C$, letting $r=(n/\Cr{c:choicen2})^{1/\alpha}$, 
\begin{equation*}
\begin{split} 
    \P\big(|\K^a_{\Cr{C:localuniq}r}|\geq n\big)
    &\ge\P(|\K^{-1}_{\Cr{C:localuniq}r}|\geq \Cr{c:choicen2}r^{\alpha})
    \cdot \exp \bigg\{ - \frac{(a+1)^2 {\rm cap}(B(\Cr{C:localuniq}r)) + 2/{\rm e}}{2\P(|\K^{-1}_{\Cr{C:localuniq}r}|\geq \Cr{c:choicen2}r^{\alpha})} \bigg\}
    \\&\geq c\exp\{-C(a\vee1)^2r^{\nu}\}\geq c\exp \big \{-C(a\vee1)^2n^{\frac\nu\alpha} \big \},
\end{split}
\end{equation*} 
where we used \eqref{eq:simpleK-1} and \eqref{eq:capBallBd} in the second inequality.
The lower bound in \eqref{eq:easyoffcriticalbounds} for $n\geq C$ is then a direct consequence of the inclusion $\K^{a}_{\Cr{C:localuniq}r}\subset\K^{a}$. The case $1\leq n< C$ is readily taken care of by adapting the constant $C$ appearing in the lower bound of \eqref{eq:easyoffcriticalbounds}.
\end{proof}

Let us conclude this section by explaining how we will improve the proof of the lower bound in \eqref{eq:easyoffcriticalbounds} in the next two sections to obtain the better lower bound from \eqref{eq:tailEstVolgen} for small values of $a>0$. The main improvement is that we will use the change of measure formula \eqref{eq:entropy} with $b=a$ instead of $b=1$. This requires obtaining good bounds on $|\K^{-a}_r|$ for small values of $a>0$,  which are explicit in $a$, instead of our simple bound \eqref{eq:simpleK-1} on $|\K^{-1}_r|$. A first step in this direction will be to obtain good bounds on the volume $\LV_r^a$ of the largest cluster in a ball of size $r$ at level $-a$, which is the main content of Section~\ref{sec:volneg}, see in particular Proposition~\ref{pro:manyMeso}.  These findings are also at the root of our results concerning $\LV_r^a$ itself, cf.~\eqref{eq:critlargestvol-Zd2} and Theorem~\ref{the:mainMra}. 

The next step would in principle be to bound from below the probability that $0$ is connected to this largest cluster, so as to make $|\K_r^{-a}|$ large. However, this probability will depend on $a$ and $r$ (since it converges to $0$ as $a\rightarrow0$  and $r \to \infty$), and this dependence would appear unfavorably in the exponential of the change of measure \eqref{eq:entropy}. We solve this problem in Section~\ref{sec:den}, see in particular \eqref{eq:LBcond0}, \eqref{eq:AKary2}, \eqref{eq:clustCardLB} and Lemma~\ref{lem:boundParF}, by following ideas from \cite[Section~6]{DrePreRod5}.  Rather than connecting to $0$ directly, a critical cost will be incurred for the cluster of $0$ to reach a desired capacity within a certain linear size (its realization will play the role of $K$ in \eqref{eq:K-ass}), thus rendering the desired connection cost-efficient.

Then, in order to connect this cluster of $0$ to the largest cluster in $B(r)$ at level $-a$, we will use a random interlacement trajectory, and thus in view of \eqref{eq:defIu}, we will not only need that this cluster has a large volume, but also a large capacity, see \eqref{eq:ALBsplit} for details. Note that this is in fact the case for the cluster of random interlacements constructed in the proof of Proposition~\ref{pro:easyoffcriticalbounds}, by combining \eqref{eq:Iu-capbound} and the FKG inequality, and this large capacity carries over to $\K^{a}_{\Cr{C:localuniq}r} (\subset\K^{a})$ after application of the isomorphism and change of measure.

\section{Volume above negative levels}
\label{sec:volneg}

Recall from the beginning of  Section~\ref{sec:pre} that we always tacitly work under the assumption that $(\G,\lambda)$ satisfies
\eqref{eq:ellipticity}, \eqref{eq:intro_sizeball} and \eqref{eq:intro_Green}; see also Section~\ref{sec:mainresults} for notation.  For technical reasons that will become apparent in the next section (cf.~also the discussion at the end of the previous section), we work in a more general setup, i.e.~we consider the free field (with law $\PK$) on the cable system $\tilde{\G}_K$ instead of $\tilde{\G}$, with $K \subset \tilde{\G}$ as in \eqref{eq:K-ass}. This corresponds to the graph on which the associated diffusion is killed on hitting $K$, see the beginning of Section~\ref{sec:pre} for details.

In this section we focus on ${\K}_{\tilde B}^{-a}(y)$, the cluster of $y$ in $\tilde B \cap \{ \varphi \geq -a\}$ for (small) $a>0$,  see \eqref{eq:Kr} for notation.
The main result of this section is that there exists $y\in\widetilde{B}\cap G$ with $\widetilde{B}=\widetilde{B}(x,r)$ such that both $|{\K}_{\tilde B}^{-a}(y)|$ and $\mathrm{cap}_{\tilde{\G}_K}\big({\K}^{-a}_{\tilde B}(y)\big)$, see \eqref{eq:vol} and \eqref{eq:cap} for notation, are large with high enough probability, even for small values of $a$. Recall the definition of $q(\cdot)$ from \eqref{eq:qPsiBd}.

\begin{Prop}\label{pro:manyMeso} There exist constants  $\Cl[c]{c:q}$, $\Cl[c]{c:LBhatK}$, $c \in (0,1)$ and $\Cl{C:LBLneg},\Cl{C:dxK}<\infty$ such that for all $K\subset \tilde{B}(r)$ satisfying \eqref{eq:K-ass}, and all $a\in{(0,1]}$, $r\geq  C $ and $x\in G_K$ with $rq(\Cr{c:q}r)^{-\frac2\nu}\geq \Cr{C:LBLneg}a^{-2/\nu}$, $d(x,K)\geq \Cr{C:dxK}r$ and $\tilde B =\tilde B(x,r)\subset\tilde{\G}_K$, 
the following hold:
\begin{equation} \label{eq:LBL2} 
        \PK \bigg(\exists\, y\in B(x,r):\, \big| {\K}^{-a}_{\tilde B}(y)\big| \ge \frac{\Cr{c:LBhatK} a r^\alpha}{q(\Cr{c:q}r)^{2}}, \, \mathrm{cap}_{\tilde{\G}_K}\big({\K}^{-a}_{\tilde B}(y)\big)\geq \frac{cr^{\nu}}{q(\Cr{c:q}r)^{2}} \bigg) \ge 
    cq(\Cr{c:q}r)^{-4}.
\end{equation}
\end{Prop}

The bound \eqref{eq:LBL2} is mainly interesting when $\sup q = \sup_r q(r) < \infty$, which is for instance the case on $\Z^d$, $d\in\{3,4,5\}$, cf.~\eqref{eq:boundq-new}, and will be an essential ingredient in the proofs of both Theorem~\ref{thm:volLB} and Theorem~\ref{the:mainMra}.  Note in particular that in this case \eqref{eq:tailEstlargestVolgen} is a direct consequence of \eqref{eq:LBL2}  $a<0$ and $r\geq C|a|^{-2/\nu}$. When the assumption $\sup q<\infty$  is not satisfied, one can actually still obtain a result similar to \eqref{eq:LBL2} with $q$ replaced by $1$ therein, but only under the additional assumption $r\geq a^{-C}$ for a large constant $C$ (and in particular a priori larger than $2/\nu$), see Remark~\ref{rk:endsection4},\ref{rk:casesupq=infty}.

The rest of this section is dedicated to the proof of Proposition~\ref{pro:manyMeso}. We start by controlling the number of points contained in critical level sets of the Gaussian free field with large capacity in $\tilde B$. Recall that $\mathrm{cap}_{\tilde{\G}_K}$ denotes the capacity on the graph $\tilde{\G}_K$, see \eqref{eq:cap} (see also the discussion following \eqref{eq:K-ass} regarding the metric graph $\tilde{\G}_K$), and define the event
\begin{equation}
\label{eq:defD1}
D(y; x, r,s) \stackrel{\text{def.}}{=} 
\big \{\mathrm{cap}_{\tilde{\G}_{K}}\big({\K}^0_{\tilde B}(y)\big) \geq s \big\}, \quad \tilde B=\tilde B(x,r).
\end{equation}
If the remaining parameters $r,s$ are clear from the context, we simply write $D(y)$ instead of $D(y; x,r,s)$. 
Note that for $s>0$, as will always be the case below, $D(y; x,r,s)$ is empty unless $y \in \tilde B$.
For $x \in \widetilde \G$ and $r ,s \in (1,\infty)$ we further define the set 
\begin{equation}\label{eq:Fdef}
      \mathcal{D}(x,r,s) \stackrel{\text{def.}}{=} \big\{ y \in B(x,r) \, : \, D(y; x,r,s)\text{ occurs} \big\},
\end{equation}
 and then write
\begin{equation} \label{eq:Edef}
     \mathcal E (x,r,s) \stackrel{\text{def.}}{=} \big\{|\mathcal{D}(x,r,s)| \ge \tfrac12 \EK[|\mathcal{D}(x,r,s)|]\big\}.
\end{equation}

\begin{Lemme}\label{lem:manyPtsInTypCl}
For all $s\in (1,\infty)$,  $r\in{(2,\infty)}$ with
\begin{equation} \label{eq:rCond}
 s\leq  s_{{r}/{2}}, \quad \text{where } s_r \stackrel{\textnormal{def.}}{=} \Cr{C:lawcaploc}r^{\nu}q(r)^{-2}  \text{ (see Lemma~\ref{lem:capbox} for~$\Cr{C:lawcaploc}$),} 
\end{equation}
and for all  $K\subset\tilde{\G}$ and $x\in{G_K}$ satisfying  \eqref{eq:K-ass} and $\tilde{B}(x,r)\subset \tilde{\G}_K$, one has
\begin{equation}
\label{eq:secondmoment}
     \EK \big[| \mathcal{D}(x,r,s)|\big]\geq c r^{\alpha}s^{-\frac12}\text{ and }\EK \big[| \mathcal{D}(x,r,s)|^2\big]  \le C r^{2\alpha}s^{-1}.
\end{equation}
In particular,
\begin{equation} \label{eq:ELB}
   \PK \big(\mathcal E(x,r,s )\big) \ge c.
\end{equation}

\end{Lemme}
\begin{proof} Throughout the proof we tacitly assume that $x\in G \setminus {B(0,2r)}$, as appearing above \eqref{eq:secondmoment}. By application of Lemma~\ref{lem:capbox}, which is in force owing to the condition \eqref{eq:rCond}, one has that for all $y\in{B(x,  r/2 )}$, abbreviating $D(y)=D(y; x,r,s)$,
\begin{equation}\label{eq:clustCapAsymp} 
\begin{split}
\PK ( D(y) )
&\geq\PK\big( \mathrm{cap}_{\tilde{\G}_{K}}\big({\K}^0_{ \tilde B(y,r/2) }(y)\big)\geq s  \big)
\stackrel{\eqref{eq:clustCapLocal}}{\ge} cs^{-1/2}
\end{split}
\end{equation}
where the last equality holds true owing to \eqref{eq:rCond}, which implies that the condition needed for \eqref{eq:clustCapLocal} to apply is satisfied. We can therefore  use \eqref{eq:clustCapAsymp} to bound the first moment from below upon restricting the relevant sum to $ y \in{B(x, r/2 )}$ by
\begin{align} \label{eq:firstMom}
\begin{split}
\EK \big[| \mathcal{D}(x,r,s)|\big] 
&=  \sum_{y \in B(x,r) }\PK (D(y)) 
 \ge c r^{\alpha}s^{-1/2},
 \end{split}
\end{align}
where we used \eqref{eq:intro_sizeball} and \eqref{eq:lambdabounded} in the last inequality. This yields the first item in \eqref{eq:secondmoment}. 

With a view towards to the second moment, we start by considering $\PK(D(y), D(z))$ for $y,z \in B(x,r)$ with $y \neq z$. Let $d_{\text{gr}}$ denote the graph distance on $G$, and for $z\in{G}$ and $K'\subset \tilde{\G}$ we abbreviate $d_{\text{gr}}(z,K')=\inf_{y\in{K'\cap G}}d_{\text{gr}}(z,y)$. Using monotonicity of ${\K}^{0}_{\tilde B}(y)$ in $\tilde{B}$ to replace ${\K}^{0}_{\tilde B}(y)$ by $ {\K}^{0}(y) \equiv \K^0_{\tilde \G}(y)$, and with $\widehat{D}$ denoting the same event as in \eqref{eq:defD1} but with $\mathcal{K}^0(y)$ in place of ${\K}^{0}_{\tilde B}(y)$, we then have, conditioning on ${\mathcal K}^0(y)$ (see \eqref{eq:Markov1} for the relevant $\sigma$-algebra $\mathcal{A}_{{\mathcal K}^0(y)}^+$),
\begin{align} \label{eq:UBProbExc}
\begin{split}
 &\PK(D(y), D(z)) \leq \PK(\widehat{D}(y), \widehat{D}(z))
 = \EK \Big[1_{\widehat{D}(y)}\PK\big(\widehat{D}(z) \, | \, \mathcal{A}_{{\mathcal K}^0(y)}^+\big) \Big]\\
 &\le\PK\big(d_{\text{gr}}(z, {\mathcal K}^0(y))\leq  1\big) \\
 &\quad +
  \EK\Big[1\big\{d_{\text{gr}}(z, {\mathcal K}^0(y) )> 1,\, {\rm cap}_{\tilde{\G}_K}({\mathcal K}^0(y)) \ge s\big\} \PK\big({\rm cap}_{\tilde{\G}_K}({\mathcal K}^0(z)) \ge s  \, \big| \, \mathcal{A}_{{\mathcal K}^0(y)}^+\big) \Big].
\end{split}
\end{align}
Now, observe that for $z\neq y$
$$
\PK\big(d_{\text{gr}}(z, {\mathcal K}^0(y))\leq 1\big) \le \sum_{z' \sim z}\PK\big(z' {\longleftrightarrow} y\text{ in }\{\phi\geq0\}\big),
$$
which, due to \cite[Prop. 5.2]{MR3502602} and \eqref{eq:intro_Green}, is bounded from above by \( Cd(y,z)^{-\nu} \) (note that \eqref{eq:ellipticity} implies that $\G$ has bounded degree). To deal with the term in the last line of \eqref{eq:UBProbExc}, we
apply the strong Markov property for the Gaussian free field, see \eqref{eq:Markov2},  either for the set $\K^0(y)$ or $\K^0(z)$, as follows. On the
event $d_{\text{gr}}(z, {\mathcal K}^0(y))> 1$, one has that  $z\notin{\K^0(y)}$ and in fact that $\K^0(y)$ and $\K^0(z)$ are disjoint. Recall now that $\tilde \G_U$ corresponds to the \emph{unbounded
component} of $\tilde \G \setminus U$, cf.~below \eqref{eq:K-ass}. Note that, under $\PK$, either $z$ (hence $\K^0(z)$) belongs to the unbounded component of the set $\tilde{\G}_K \setminus \K^0(y)$ (call this event $A_z$), in which case
$\K^0(z) \subset \tilde{\G}_{K\cup \K_0(y)}$, or otherwise $A_y$ occurs (whence $\K^0(y) \subset \tilde{\G}_{K\cup \K_0(z)}$). We will assume that $A_z$ (which is $\mathcal{A}_{{\mathcal K}^0(y)}^+$-measurable) occurs in the sequel and condition on $\mathcal{A}_{\K^0(y)}^+$; if instead $A_y$ occurs one conditions on $\mathcal{A}_{{\mathcal K}^0(z)}^+$ instead and applies a similar reasoning as below.

Thus by \eqref{eq:Markov2}, on the event $z\notin{\K^0(y)}$ implied by $d_{\text{gr}}(z, {\mathcal K}^0(y))> 1$  and on $A_z$, since $h_{\K^0(y)}^{K}=0$ on $\tilde{\G}_K\setminus \K^0(y)$, the set $\K^0(z)$ has the same law under $\PK(\cdot\,|\, \mathcal{A}_{{\mathcal K}^0(y)}^+)$ as under $\PKK{K\cup{\K}^0(y)}$. Observing in addition that $\mathrm{cap}_{\tilde{\G}_{K'}}(A)\geq \mathrm{cap}_{\tilde{\G}_K}(A)$ if $K\subset K'$, we deduce that on the event $\{d_{\text{gr}}(z, {\mathcal K}^0(y) )> 1\} \cap A_z$,
 \begin{equation*}
    \PK\big({\rm cap}_{\tilde{\G}_K}({\mathcal K}^0(z)) \ge s  \, \big| \, \mathcal{A}_{{\mathcal K}^0(y)}^+\big) \leq \PKK{K\cup{\K}^0(y)}\big({\rm cap}_{\tilde{\G}_{K\cup{\K}^0(y)}}({\mathcal K}^0(z)) \ge s  \big)\leq Cs^{-\frac12},
 \end{equation*}
 where the last inequality follows from \eqref{eq:capTail} applied to the graph $\tilde{\G}_{K\cup{\K}^0(y)}$ and the inequalities $$g_{\tilde{\G}_{K\cup{\K}^0(y)}}(z,z)\geq 1/\lambda_z^{K\cup{\K}^0(y)}=1/\lambda_z\geq c$$ on the event $d_{\text{gr}}(z, {\mathcal K}^0(y) )> 1$, see  below \eqref{eq:K-ass} and \eqref{eq:lambdabounded} regarding $\lambda^{K\cup{\K}^0(y)}_{\cdot}$. Continuing with \eqref{eq:UBProbExc}, substituting the above estimates and applying \eqref{eq:capTail} once again, we obtain altogether that for $y,z \in B(x,r)$ with $y\neq z$,
\begin{align} \label{eq:2capLargeUB}
   \PK\big(D(y), D(z)\big)
    &\le C(d(y,z)\vee1)^{-\nu}+ C s^{-1}.
\end{align}
Summing \eqref{eq:2capLargeUB} over $y\neq z$ and using the trivial bound $\PK\big(D(y)\big)\leq 1$ thus yields that
\begin{equation} \label{eq:secMomDec}
\begin{split}
 \EK \Big[| \mathcal{D}(x,r,s)|^2\Big]  
&=  \sum_{\substack{y,z \in B(x,r)}}\PK\big(D(y), D(z)\big)
\\&\le C\sum_{y \in B(x,r)} \bigg(\black 1+\sum_{k =1}^{\lceil\log_2(r)\rceil}  |B(y,2^k) \setminus B(y,2^{k-1})| \cdot \big(2^{-k\nu} +  s^{-1}\big)\bigg)
\\&\le C\sum_{y \in B(x,r)} \bigg(\black 1+\sum_{k =1}^{\lceil\log_2(r)\rceil} 2^{\alpha k} \big(2^{-k\nu} + s^{-1} \big)\bigg)
\\&\le C \sum_{y \in B(x,r)} (r^{\alpha-\nu} + r^{\alpha} s^{-1}) \le C r^{2\alpha}s^{-1},
\end{split}
\end{equation}
where we used \eqref{eq:intro_sizeball} and \eqref{eq:lambdabounded} in the two last inequalities, as well as the assumption $\alpha\geq \nu+2$ (see below \eqref{eq:intro_Green}), and used the assumption \eqref{eq:rCond} in the last inequality (recall from \eqref{eq:boundq} that  $q(\cdot)\geq c$). 
This finishes the proof of \eqref{eq:secondmoment}, which directly yields \eqref{eq:ELB} by means of the Paley--Zygmund inequality.
\end{proof}

 In order to deduce Proposition~\ref{pro:manyMeso} from Lemma~\ref{lem:manyPtsInTypCl}, we will use the isomorphism with random interlacements stated in \eqref{eq:isom}, which becomes pertinent at negative levels $-a$ for a given $a >0$, with $\sqrt{2u}=a$. 

For a set $I \subset \tilde \G$ we denote those elements of $\mathcal{D}(x,r,s)$ whose sign clusters in $\tilde B=\tilde B(x,r)$ intersect  $I$ by
\begin{equation} \label{eq:signConn}
\mathcal H(x,r, s,I) \stackrel{\text{def.}}{=}  \big\{y \in \mathcal{D}(x,r, s) \, : \,  \K^0_{ \tilde B}(y) \cap I \ne \emptyset \big\}.
\end{equation}
We also denote by $\hat{\K}_{\tilde B}^{0}$ the set ${\K}_{\tilde B}^{0}(y)$ with the maximal capacity with respect to the graph $\G_K$ among $y\in{B(x,r)}$, and define under $\PK \otimes \OPK$ the set
\begin{equation}\label{eq:Ihat}
\hat{\I}^u_{\tilde B}= \left\{\text{\begin{minipage}{0.6\textwidth} \centering the union of the forward parts of random interlacement trajectories in $\I^u$ which intersect $\hat{\K}_{\tilde B}^{0}$, started when entering $\hat{\K}_{\tilde B}^{0}$ and killed when first exiting $\tilde{B}(x,\Cr{C:exit}r)$ \end{minipage}}\right\},
\end{equation}
(with $\Cr{C:exit}$ as introduced in \eqref{eq:hittingvscap}).

In view of the isomorphism \eqref{eq:isom}, with a view towards deducing Proposition~\ref{pro:manyMeso}, it will be enough to obtain bounds similar to \eqref{eq:LBL2} but for the connected component in  $\mathcal{C}_{a^2/2}\cap \tilde{B}(x,r)$ with the largest volume instead of $\widehat{\K}^{-a}_{\tilde B}$, which is the goal of the next lemma. 

\begin{Lemme}\label{lem:manyMeso} 
For all $u\in{(0,1)}$ and $r\geq C$ satisfying 
 $u^{-1}\leq s_{r/2}$ (cf.~\eqref{eq:rCond}), all $K\subset\tilde{B}(r)$  as in \eqref{eq:K-ass}, and all 
 $x\in G_K$ with $d(x,K)\geq 2\Cr{C:exit}r$, 
 \begin{equation} \label{eq:LBL21}
\PK \otimes \OPK\bigg( \big| \mathcal H\big(x,r, u^{-1}, \hat{\I}^u_{\tilde B}\big) \big| \ge \frac{c\sqrt{u}r^\alpha}{q(r/2)^{2}}, \ \mathrm{cap}_{\tilde{\G}_K}( \hat{\K}_{\tilde B}^{0})\geq \frac{cr^{\nu}}{q(r/2)^{2}} \bigg) \ge cq(r)^{-4}.
    \end{equation}
\end{Lemme}

\begin{proof}
We use a second moment argument again, recall $s_r$ from \eqref{eq:rCond} and that $s_{r/2} \geq u^{-1}$ by our assumption on $r$.  First note that by \eqref{eq:defD1}, \eqref{eq:Edef} and \eqref{eq:secondmoment}, on the event ${\mathcal E}(x,r,s_{r/2})$, if $r\geq C$ then  using that $\lim_{r \to \infty} r^{\alpha-\frac\nu2}=\infty$,  there is at least one $y\in{B(x,r)}$ such that $\mathrm{cap}_{\tilde{\G}_K}(\K^0_{\tilde{B}})\geq s_{r/2}$. In particular, the number of trajectories under $\OPK$ of interlacements at level $u$ hitting $\hat{\K}^0_{\tilde B}$ that constitute $\hat{\I}^u_{\tilde B}$ in \eqref{eq:Ihat} is a Poisson variable with parameter larger than $ us_{r/2}= cur^{\nu}q(r/2)^{-2}$. By the thinning property of Poisson variables, for any $y\in{\mathcal{D}(x, r, u^{-1})}$  the number of such trajectories which hit  ${\K}^0_{\tilde{B}}(y)$ before first exiting $\tilde{B}(x,\Cr{C:exit}r)$ is a Poisson variable with parameter larger than 
\begin{align*}
cur^{\nu}q(r/2)^{-2}\inf_{z\in{\hat{\K}^{0}_{\tilde B}}}\PPK_z\big(H_{{\K}^0_{\tilde{B}}(y)}<H_{\tilde{B}(x,\Cr{C:exit}r)^{\mathsf c}} \big)
&\geq c' uq(r/2)^{-2}\mathrm{cap}_{\tilde{\G}_K}\big({\K}^0_{\tilde{B}}(y) \big)\\
&\geq c' q(r/2)^{-2},
\end{align*}
where in the first inequality we used \eqref{eq:hittingvscap}, and we used \eqref{eq:defD1} and \eqref{eq:Fdef}, by which $\mathrm{cap}_{\tilde{\G}_K}({\K}^0_{\tilde{B}}(y)) \geq u^{-1}$ for any $y\in{\mathcal{D}(x, r, u^{-1})}$, to obtain the last inequality. We deduce, using the previous observations along with \eqref{eq:Edef} and \eqref{eq:secondmoment}, that for $r \ge C$,
\begin{align}
\label{eq:firstmoment2}
\begin{split}
    \EK \otimes \OEK&\Big[ 1\big\{\mathrm{cap}_{\tilde{\G}_K}( \hat{\K}_{\tilde B}^{0})
    \geq s_{r/2}\big\} \cdot \sum_{y \in \mathcal{D}(x,r,u^{-1})} 1\{{ \mathcal K}^0_{\tilde B}(y) \cap  \hat{\I}^u_{\tilde B} \ne \emptyset \} \Big] \\
    &\ge \EK \Big[ 1\{{\mathcal E}(x,r,s_{r/2})\} \cdot \sum_{y \in \mathcal{D}(x,r,u^{-1})} \OPK\big({ \mathcal K}^0_{\tilde B}(y) \cap  \hat{\I}^u_{\tilde B} \ne \emptyset  \big) \Big] \\
    &\ge \PK \big( {\mathcal{E}}(x,r,s_{r/2}) \cap \mathcal{E}(x,r,u^{-1})\big)cr^{\alpha}u^{\frac12} \big( 1 - \exp\{-c q(r/2)^{-2}\} \big).
\end{split}
\end{align}
In view of \eqref{eq:defD1}, \eqref{eq:Fdef} and \eqref{eq:Edef}, the event $\mathcal{E}(x,r,s)$ is increasing  in $\varphi$, and thus decreasing in $s$. Therefore, by the FKG inequality and \eqref{eq:ELB}, since $s_{r/2} \geq u^{-1}$ by assumption, one can bound the probability in the last line of \eqref{eq:firstmoment2} from below by a constant. Let us abbreviate $A=\big\{\mathrm{cap}_{\tilde{\G}_K}( \hat{\K}_{\tilde B}^{0})\geq s_{r/2}\big\}$. It thus follows from \eqref{eq:firstmoment2}, \eqref{eq:signConn}, recalling also that $q(\cdot)\geq c$, that 
\begin{equation*}
\begin{split}
\EK \otimes \OEK\Big[ \big|  \mathcal H(x,r, u^{-1}, \hat{\I}^u_{\tilde B}) \big| \cdot 1_A  \Big] \ge cr^\alpha u^{\frac12}q(r/2)^{-2}.
\end{split}
\end{equation*}
For the second moment we can crudely upper bound in view of \eqref{eq:signConn}
\begin{equation}
\label{eq:secondmoment3}
\EK \otimes \OEK\Big[ \big|  \mathcal H(x,r, u^{-1}, \hat{\I}^u_{\tilde B}) \big|^2 \cdot 1_A  \Big]\leq \EK \big[| \mathcal{D}(x,r,u^{-1})|^2\big]
\le Cur^{2\alpha},
\end{equation}
where the last inequality is \eqref{eq:secondmoment}. Applying the Paley-Zygmund inequality under $\PK \otimes \OPK(\cdot\,|\,A)$ then yields that
\begin{align*}
\PK \otimes \OPK&\Big(| \mathcal H(x,r, u^{-1}, \hat{\I}^u_{\tilde B}) |\ge c \sqrt{u}q(r/2)^{-2}r^\alpha\,\big|\,A \Big)
\\ &\geq \frac{\EK \otimes \OEK\big[|( \mathcal H(x,r, u^{-1}, \hat{\I}^u_{\tilde B}) |\cdot 1_A\big]^2}{\EK \otimes \OEK\big[|( \mathcal H(x,r, u^{-1}, \hat{\I}^u_{\tilde B}) |^2 \cdot 1_A \big]\PK(A)}
\geq \frac{cur^{2\alpha} q(r/2)^{-4}}{Cur^{2\alpha}\PK(A)},
\end{align*}
from which \eqref{eq:LBL21} follows.
\end{proof}

As we now explain, the inequality \eqref{eq:LBL2} is a direct consequence of \eqref{eq:LBL21} and the isomorphism \eqref{eq:isom}. 
\begin{proof}[Proof of Proposition~\ref{pro:manyMeso}]
    Recalling
$\mathcal H(x,r, s,I)$ from \eqref{eq:signConn}, let us define for $I\subset\tilde{\G}$ the set
    \begin{equation}
    \label{eq:tildeH}
\begin{split} 
    \tilde{\mathcal{H}}(x,r,s,I)\stackrel{{\rm def.}}{=} \bigcup_{y\in{{\mathcal{H}}(x,r,s,I)}}\K^0_{\tilde B}(y).
\end{split}
\end{equation}
Then, since  $\hat{\I}^u_{\tilde B}$ is a connected subset of $\tilde B(x,\Cr{C:exit}r)$ by definition (see \eqref{eq:Ihat}), the set $\hat{\I}^u_{\tilde B}\cup\tilde{\mathcal H}(x,r, u^{-1}, \hat{\I}^u_{\tilde B})$ is a connected subset of $B(x,\Cr{C:exit}r)$, see \eqref{eq:defD1},\eqref{eq:Fdef},  and \eqref{eq:signConn}, which by definition is included in $\mathcal{C}_u$, see above \eqref{eq:isom}. If now $rq(\Cr{c:q}r)^{-\frac2\nu}\geq \Cr{C:LBLneg}a^{-2/\nu}$, $d(x,K)\geq 2\Cr{C:exit}r$ and $\Cr{c:q}= 1/\Cr{C:exit}$, 
since $\mathcal{H}(x,r,u^{-1},\hat{\I}^u_{\tilde B})\cup \hat{\K}_{\tilde B}^{0}\subset\tilde{\mathcal{H}}(x,r,u^{-1},\hat{\I}^u_{\tilde B})$, after a change of variable for (i.e.~rescaling of) $r$, \eqref{eq:LBL2} follows directly from \eqref{eq:isom}, \eqref{eq:LBL21} for $u=a^2/2$, and monotonicity of the capacity,  noticing that the condition $u^{-1}\leq s_{r/2}$ appearing in Lemma~\ref{lem:manyMeso}  is satisfied by assumption on $r$. 
\end{proof}

\begin{Rk}
\label{rk:endsection4}
\begin{enumerate}[1)]
\item \label{rk:casesupq=infty} Under the conditions of Proposition~\ref{pro:manyMeso}, if additionally $r\geq Ca^{-C'}$, one can actually improve \eqref{eq:LBL2} to obtain
        \begin{equation} \label{eq:LBL4}
\PK \Big(\exists\, y\in G:\, \big| {\K}^{-a}_{\tilde B}(y)\big| \ge \Cr{c:LBhatK} ar^\alpha, \, \mathrm{cap}_{\tilde{\G}_K}\big({\K}^{-a}_{\tilde B}(y)\big)\geq cr^{\nu} \Big)\ge 
    c.
    \end{equation}
To prove \eqref{eq:LBL4}, one can first replace the bound \eqref{eq:LBL21} by
\begin{equation} \label{eq:LBL22}
    \PK \otimes \OPK\big( | \mathcal H(x,r, u^{-1}, \I^u) | \ge c \sqrt{u}r^\alpha \big) \ge c,
\end{equation}
which is a consequence of the Paley-Zygmund inequality: one can use \eqref{eq:defIu}, \eqref{eq:defD1} and \eqref{eq:Fdef} to show that $\K_{\tilde{B}}^0(y)$  has a constant probability to intersect $\I^u$ for any $y \in \mathcal{D}(x,r,u^{-1})$, which lets us bound the first moment of $| \mathcal H(x,r, u^{-1}, \I^u) |$ by $cu^{1/2} r^\alpha$ using \eqref{eq:secondmoment}, and the second moment can be bounded from above similarly as in \eqref{eq:secondmoment3}. Contrary to the argument below \eqref{eq:tildeH}, the set ${\I}^u\cup\tilde{\mathcal H}(x,r, u^{-1}, {\I}^u)$ is however not necessarily connected in $B(x,Cr)$. This is nonetheless the case under the event $\textnormal{LocUniq}_{u,r}(x)$ from \eqref{eq:lu-def},  which occurs at the same time as the event in \eqref{eq:LBL4} with constant probability by \eqref{eq:boundlocaluniq}, under the additional assumption $r\geq Cu^{-C}$.  By \eqref{eq:Iu-capbound} and the FKG inequality, the previous set has capacity (in $\tilde\G_K$) at least $cr^{\nu}$ with sizeable probability if $r\geq Cu^{-C}$, 
 and using the isomorphism \eqref{eq:isom} similarly as below \eqref{eq:tildeH} then yields \eqref{eq:LBL4}.
\item \label{rk:casealpha>2nu}
The previous proof of \eqref{eq:LBL4} is arguably easier than the proof of \eqref{eq:LBL2}, as it does not require to introduce the complicated set $\hat{\I}_{\tilde{B}}^u$ from \eqref{eq:Ihat}, and \eqref{eq:LBL4} does not require any information on the value of the function $q$ from \eqref{eq:qPsiBd}. However, it requires $r\geq Ca^{-C'}$ without good control on the value of $C'$, which will be detrimental when trying to prove \eqref{eq:tailEstVolgen} and \eqref{eq:tailEstlargestVolgen} near criticality, see Remark~\ref{rk:final},\ref{rk:boundinanydimension} for details.  An exception is when $\alpha>2\nu$, in which case this condition can be replaced by the requirement that $rq(\Cr{c:q}r)^{-\frac2\nu}\geq \Cr{C:LBLneg}a^{-2/\nu}$ as in \eqref{eq:LBL2}, which offers an alternative way to eventually derive Theorem~\ref{T:offcritvol-Zd-LB} when $\alpha > 2\nu$, and is closer to the strategy from \cite{DrePreRod5}. The improved condition on $a$ and $r$ follows on account of \eqref{eq:c4explicit}, which allow to apply both \eqref{eq:boundlocaluniq} and \eqref{eq:Iu-capbound} when $r \geq Ca^{-2/\nu}$. A similar remark can be made when $\alpha=2\nu$ (for instance on $\Z^4$), but with additional logarithmic corrections (cf.~for instance \cite[(1.4)]{prevost2023passage} for best available results regarding \eqref{eq:boundlocaluniq} in this case). 
\end{enumerate}
\end{Rk}

\section{Proof of main results} \label{sec:den}

In this section we finish the proof of Theorems~\ref{thm:volLB} and~\ref{the:mainMra} using Proposition~\ref{pro:manyMeso}, and deduce Corollary~\ref{C:offcritvol-mom}. This completes the proof of all results from Section~\ref{sec:mainresults}. Recall that the results stated in the introduction follow by combining these (as explained in detail below each result in the introduction) and considering the special case $\G=\Z^\alpha$, $\alpha\geq 3$, wit unit weights and $d(\cdot,\cdot)$ the graph, or Euclidean, distance, for which $\nu=\alpha-2$. 

\bigskip
We start with the simpler: 
\begin{proof}[Proof of Theorem~\ref{the:mainMra}]
We first show \eqref{eq:tailEstlargestVolgen}. When $a\in{(-1,0)}$ and $r\geq \Cl{C:boundarlargest}a^{-2/\nu}$ for sufficiently large $\Cr{C:boundarlargest}$, \eqref{eq:tailEstlargestVolgen} is an immediate consequence of Proposition~\ref{pro:manyMeso}, applied with $K=\emptyset$, and the assumption $\sup q<\infty$. When $a\in{(0,1)}$ and $r\geq \Cr{C:boundarlargest}a^{-2/\nu}$, \eqref{eq:tailEstlargestVolgen} follows from the entropy bound \eqref{eq:entropy} with $K=\emptyset$, $b=2a$, $A=B(r)$ and $E=\{\LV_r^a\geq n\}$, which imply that 
\begin{align} \label{eq:largestclustCardLB}
\begin{split}
    \P\big(\LV^a_r\geq n\big)
    &\ge\P\big(\LV^{-a}_r\geq n\big)
    \cdot \exp \Big\{ - \frac{(2a)^2 {\rm cap}(B(r)) + 2/{\rm e}}{2\P\big(\LV^{-a}_r\geq n\big)}\Big\} \\ 
    &\geq \exp\{-Ca^2r^{\nu}\},
\end{split}
\end{align}
where we used \eqref{eq:capBallBd} and \eqref{eq:tailEstlargestVolgen} for $-a$ in the last inequality. When $a\in{[0,1]}$, $r\geq \Cr{C:boundarlargest}$ and $r< \Cr{C:boundarlargest}a^{-2/\nu}$ (with $0^{-2/\nu}=+\infty$), one notices that if $a'=(r/\Cr{C:boundarlargest})^{-\nu/2}$, then the conclusions of \eqref{eq:tailEstlargestVolgen} have already been proved for the pair $(a',r)$. But since $a' \leq a$ by assumption, one has $\LV_r^a\geq \LV_r^{a'}$ by monotonicity, and \eqref{eq:tailEstlargestVolgen} for the pair $(a,r)$ follows. When $a\in{[0,1]}$ and $r\leq \Cr{C:boundarlargest}$, \eqref{eq:tailEstlargestVolgen} follows readily from the previous cases upon possibly adapting the constants. Finally, when $a\in{(-1,0)}$ and $r< \Cr{C:boundarlargest}a^{-2/\nu}$, \eqref{eq:tailEstlargestVolgen} follows from \eqref{eq:tailEstlargestVolgen} for $a=0$ and monotonicity. 

We now explain how to deduce \eqref{eq:tailEstlargestVolgena=0} from \eqref{eq:tailEstlargestVolgen} for $a=0$ and the isomorphism with loop soups. Let us first briefly recall the latter. We denote by $\mathcal{L}$ the loop soup on $\tilde{\G}$ with intensity $1/2$ from \cite[Section~2]{MR3502602}, which is a Poisson point process on the space of loops, that is, continuous paths on $\tilde{\G}$ starting and ending at the same point. Two points are in the same cluster of loops if and only if they are connected by a finite path of loops in $\mathcal{L}$ intersecting each other. We further attribute to each cluster of loops a sign $\pm$ with probability $1/2$, independently for each cluster.  By the isomorphism with loop soups \cite[Proposition~2.1]{MR3502602}, up to extending the probability space $\P$, one can couple the Gaussian free field $\phi$ on $\tilde{\G}$ with a loop soup such that $\{x\in{\tilde{\G}}:\phi_x>0\}$ has the same law as the clusters of loops in $\mathcal{L}$ with positive sign.

We now say that two points $x,y$ are in the same cluster of loops in $\tilde{B}(r)$ for $\mathcal{L}$ if and only if there is a path in $\tilde{B}(r)$ from $x$ to $y$ entirely included in the union of the loops in $\mathcal{L}$. Each such  cluster is included in a cluster of loops in $\mathcal{L}$, and inherit its random sign. The clusters of $\{x\in{\tilde{B}(r)}:\,\phi_x>0\}$ then corresponds to positive clusters of loops in $B(r)$ for $\mathcal{L}$ by the previous isomorphism. In particular, $\LV_r^0$ is equal to the cardinality of the positive cluster of loops in $\tilde{B}(r)$ for $\mathcal{L}$ with the largest cardinality. Let ${\mathcal{L}}_r$ be the point process of loops in $\mathcal{L}$ which are entirely included in $\tilde{B}(\zeta r)$, where $\zeta$ will be determined later.  Let us then denote by $\tilde{\LV}_r^0$ the cardinality of the positive cluster of loops in $\tilde{B}(r)$ for $\mathcal{L}_r$ with the largest cardinality. Then $\tilde{\LV}_r^0=\LV_r^0$ if there are no loops in $\mathcal{L}$ which hit both $\tilde{B}(r)$ and $\tilde{B}(\zeta r)^{\mathsf c}$. By \cite[Lemma~3.1]{DrePreRod8} and \eqref{eq:capBallBd}, the number of such loops is a Poisson random variable with parameter at most $C\zeta^{-\nu}$ and hence by \eqref{eq:tailEstlargestVolgen}
\begin{equation}
\label{eq:tailEstlargestVolgeninB(r)}
\begin{split} 
\P\big( \tilde{\LV}^0_r \geq cr^{\alpha-\frac\nu2} \big)&\geq \P\big( \LV^0_r \geq cr^{\alpha-\frac\nu2} \big)-(1-\exp\{-C\zeta^{-\nu}\})
\\&\geq c-(1-\exp\{-C\zeta^{-\nu}\})\geq c',
\end{split}
\end{equation}
where the last inequality is satisfied upon fixing $\zeta=\Cl{C:zeta}$ for a large enough constant $\Cr{C:zeta}$. Let us denote by $\tilde{M}_r^0(x)$ the same quantity as $\tilde{M}_r^0$, but replacing $0$ by $x$ in its definition. In other words, $\tilde{M}_r^0(x)$ is the cardinality of the positive cluster of loops in $\tilde{B}(x,r)$ with the largest cardinality, for the loops in $\mathcal{L}$ which are entirely included in $\tilde{B}(x,\Cr{C:zeta}r)$. Let $A_t=\Lambda(2\Cr{C:zeta}r)\cap B(rt)$, where $\Lambda(\cdot)$ is the set from \cite[(2.2)]{DrePreRod8}. Then $A_t$ is such that
 the balls $\tilde{B}(x,\Cr{C:zeta}r)$, $x\in{A_t}$, are disjoint, and since $\mathcal{L}$ is a Poisson point process of loops, we deduce that $\tilde{M}_r^0(x)$, $x\in{A_t}$, are independent random variable. Furthermore,  assuming w.l.o.g.\ that $t\geq4\Cr{C:zeta}$, $B(tr/2)$ is included in the union of the balls $B(x,2\Cr{C:zeta}r)$, $x\in{A_t}$ by \cite[(2.2)]{DrePreRod8}, and thus $cr^{\alpha}t^{\alpha}\leq |B(x,tr/2)|\leq |A_t| Cr^{\alpha}$ by \eqref{eq:intro_sizeball} and \eqref{eq:lambdabounded}, that is $|A_t|\geq c't^{\alpha}$. Since \eqref{eq:tailEstlargestVolgeninB(r)} as also satisfied for $\tilde{M}_r^0(x)$ by a similar proof, we deduce that for all $r,t\geq1$
 \begin{equation*}
\begin{split} 
\P \big(\tilde{\LV}^0_r(x) \leq cr^{\alpha-\frac\nu2} \text{ for all }x\in{A_t} \big)\leq c^{|A_t|}\leq \exp(-ct^{\alpha}),
\end{split}
\end{equation*}
where the last inequality follows upon fixing $t=C$ for a large enough constant $C=C(\eps)\geq4\Cr{C:zeta}$. Since $\tilde{B}(x,\Cr{C:zeta}r)\subset \tilde{B}(2tr)$ for all $x\in{A_t}$, we have $M_{2tr}^0\geq \tilde{\LV}^0_r(x)$ for all $x\in{A_t}$ by definition and the isomorphism with loop soups, which finishes the proof of \eqref{eq:tailEstlargestVolgena=0} after a change of variable for $r$ and $t$. 
\end{proof}

We now turn to the proof of Theorem~\ref{thm:volLB}. It broadly follows the general strategy from \cite[Section~6]{DrePreRod5}, see also the proof of Corollary 1.3 in \cite{DrePreRod8}, where a lower bound on the radius of $\K^a$ when $\alpha>2\nu$ was obtained. Let us now explain this strategy, appropriately modified to handle the volume of $\K^a$ and all possible values of $\alpha,\nu$. 

Recall the definition of the graph $\G_K$ from below \eqref{eq:K-ass}, so that $\tilde{\G}_K$ is the unbounded connected component of $\tilde{\G}\setminus K$, and write $\xi=\xi(a)=1\vee (2a)^{-\frac2\nu}$. As explained in \cite[(6.6) and (6.7)]{DrePreRod5}, there exists a constant $\Cl{c:hkK}\in{(1,\infty)}$ such that for any $K\subset \tilde{B}(\xi)$ as in \eqref{eq:K-ass}
\begin{equation}
\label{eq:hkK}
\begin{split} 
\tilde{\G}\setminus \tilde{B}(\Cr{c:hkK}\xi)\subset \tilde{\G}_K\text{ and }P_x(H_K<\infty)\leq\frac12\text{ for all }x\in{\tilde{B}(\Cr{c:hkK}\xi)^{\mathsf c}}.
\end{split}
\end{equation}
Recall that \(\K_r^{2a}=\K_{\tilde{B}(r)}^{2a}\) is the cluster of $0$ in $\tilde{B}(r)$ above level $2a$, see \eqref{eq:Kr} and below. Write $h_{\K^{2a}_{\xi}}=h_{\K^{2a}_{\xi}}^{\emptyset}$ for the harmonic average of $\K^{2a}_{\xi}$, see below \eqref{eq:Markov2}, then $h_{\K^{2a}_{\xi}}\geq 2aP_x(H_{\tilde{\K}_{\xi}^{2a}}<\infty)$ for all $x\in{\tilde{\G}}$. For $K$ as in \eqref{eq:K-ass}, let us abbreviate $\mathbf{h}_K(x)=(1-2P_x(H_{K}<\infty))$ for all $x\in{\tilde{\G}}$. It therefore follows from the Markov property \eqref{eq:Markov2} that 
\begin{equation}
\label{eq:consMark}
\begin{gathered} 
\{x\in{\tilde{\G}_{\K_\xi^a}}:\,\phi_x\geq a\}\text{ under $\P\big(\cdot\,|{{\K}_\xi^a}\big)$ stochastically dominates }
\\\{x\in{\tilde{\G}_{\K_\xi^a}}:\,\phi_x\geq a\mathbf{h}_{\tilde{\K}_{\xi}^{2a}}(x)\}\text{ under }\PKK{{\K}_\xi^a}.
\end{gathered}
\end{equation}
Recall the constant $\Cr{C:exit}\geq 1$ from \eqref{eq:hittingvscap}, the constant $\Cr{C:dxK}$ from  Proposition~\ref{pro:manyMeso}, the constant $\Cr{C:lawcaploc}$ from Lemma~\ref{lem:capbox}, and abbreviate $\Cl{C:exit2}=\Cr{C:exit}\vee(1+\Cr{C:dxK})$. For $K$  as in \eqref{eq:K-ass}, $a>0$, $n\geq1$, $r \geq \Cr{c:hkK} \xi$ and $h:{\tilde{\G}}\rightarrow\R,$  introduce the favorable event 
\begin{equation}
\label{eq:AKary2}
    F(K,h,r, n)\stackrel{\text{def.}}{=}\left\{\begin{array}{c}
         \text{$\exists$ $\mathfrak  C \subset  \{x\in{\tilde{\G}_K}:\phi_x\geq h(x)\} \cap \tilde B(6\Cr{C:exit2}r) $}\\
     \text{with $\mathfrak{C}$ connected, $|\mathfrak C| \ge n$, and a continuous}\\
    \text{path }\pi\text{ in }\tilde{\G}_K\cap\tilde{B}(6\Cr{C:exit2}r)\text{ from }K\text{ to }\mathfrak  C\text{ with}
    \\ \phi_x\geq  h(x) \text{ for all }x\in \pi
    \end{array}\right\}.
\end{equation}
As a consequence of \eqref{eq:consMark} and \eqref{eq:AKary2}, we obtain that for $a>0$, $n\geq1$ and $r \geq \Cr{c:hkK} \xi$
\begin{align} \label{eq:LBcond0}
    \P\big(|\mathcal K^{a}_{6\Cr{C:exit2}r}| 
    \ge n \big) \ge
    \E \Big[ 1\{{\rm cap} ( \K_\xi^{2a}) \ge \Cr{C:lawcaploc}\xi^{\nu}q(\xi)^{-2} \} \, \PKK{\K_\xi^{2a}}
    \big(F(\K_\xi^{2a},a\mathbf{h}_{\tilde{\K}_{\xi}^{2a}},r,n) \big) \Big].
\end{align}
 Let us denote by $A_{a,r}=\tilde{B}(6\Cr{C:exit2}r)\setminus \tilde{B}(\Cr{c:hkK}\xi)$ and  let $\tilde{\mathbf{h}}_{K}(x)=\mathbf{h}_K(x)-2P_x^K(H_{A_{a,r}}<\infty)$, which depends implicitly on $a$ and $r$.  Note that for each $K\subset\tilde{B}(\xi)$ as in \eqref{eq:K-ass}, the event $F(K,a\mathbf{h}_K,r,n)$ has the same same law under $\PK_{2a,A_{a,r}}$, see above \eqref{eq:entropy}, as $F(K,a\tilde{\mathbf{h}}_{K},r,n)$ under $\PK$. It thus follows from \eqref{eq:entropy} that
\begin{align} \label{eq:clustCardLB}
\begin{split}
    &\PK \big(F(K,a\mathbf{h}_K,r, n)\big)
    \ge\PK\big(F(K,a\tilde{\mathbf{h}}_{K},r, n)\big)
    \cdot \exp \Big\{ - \frac{(2a)^2 {\rm cap}_{\tilde{\G}_K}({A_{a,r}}) + 2/{\rm e}}{2\PK(F(K,a\tilde{\mathbf{h}}_{K},r,n))} \Big\}.
\end{split}
\end{align}
In view of \eqref{eq:clustCapLocal}, \eqref{eq:LBcond0} and \eqref{eq:clustCardLB}, it remains to control the various terms appearing on the right hand side of \eqref{eq:clustCardLB} for an appropriate choice of the parameters $a,b,r$ and $n$. It first follows from \eqref{eq:hkK} and a reasoning similar to \cite[(7.7)]{DrePreRod5} that if $r\geq \xi$
\begin{equation}
    \label{eq:capUB}
        {\rm cap}_{\tilde \G_K}(A_{a,r}) \le \mathrm{cap}_{\tilde{\G}}(A_{a,r})+C\xi^{\nu}\leq C'r^{\nu},
\end{equation}
where the last inequality follows from \eqref{eq:capBallBd} and the monotonicity of capacity. Recall the constants $\Cr{C:LBLneg}$ and $\Cr{c:LBhatK}$ from Proposition~\ref{pro:manyMeso},  and the constant $\Cr{c:hkK}$ from \eqref{eq:hkK}. We can finally use Proposition~\ref{pro:manyMeso} to derive the following lower bound on  $\PK(F(K,a\tilde{\mathbf{h}}_{a,r},r,n))$.

\begin{Lemme}
\label{lem:boundParF}
    For all $a\in{(0,1]}$, $K\subset \tilde{B}(\xi)$ (with $\xi=1\vee (2a)^{-2/\nu}$) as in \eqref{eq:K-ass} and $n,r\geq C$ such that 
    \begin{equation}
    \label{eq:assumpParF}
         r\geq \Cr{c:hkK}\xi,\ rq(\Cr{c:q}r)^{-\frac2\nu}\geq \Cr{C:LBLneg}a^{-\frac2\nu}\text{ and }n\leq \frac{\Cr{c:LBhatK} a r^\alpha}{q(\Cr{c:q}r)^{2}}
    \end{equation}
    we have 
\begin{equation}
\label{eq:boundParF}
\begin{split} 
    \PK \big(F(K,a\tilde{\mathbf{h}}_{K},r, n)\big)\geq c\xi^{-\nu}\mathrm{cap}(K)q(\Cr{c:q}r)^{-C}.
\end{split}
\end{equation}
\end{Lemme}

We will prove Lemma~\ref{lem:boundParF} at the end of this section, and we first explain how to finish the

\begin{proof}[Proof of Theorem \ref{thm:volLB}]
Similarly as around \eqref{eq:apos}, we only need to prove that $ \P( |\mathcal K^a| \geq n)$ satisfies the  bound \eqref{eq:tailEstVolgen} for all $a\in{[0,1]}$. We first focus on the case $n\geq \Cl{c:n>a}a^{1-\frac{2\alpha}{\nu}}$ and $a\in{(0,1]}$, where $\Cr{c:n>a}$ is a large enough constant we will fix in the proof. Let us choose
\begin{equation}
\label{eq:choicebr}
\begin{split} 
r\stackrel{\text{def.}}{=}\Big(\frac{2n\sup q^2}{\Cr{c:LBhatK} a}\Big)^{\frac1\alpha}.
\end{split}
\end{equation}
Note that  we have  $n,r\geq C$, $r\geq cn^{\frac1\alpha}a^{-\frac1\alpha}\geq c\Cr{c:n>a}^{\frac1\alpha}a^{-\frac2\nu}\geq  \Cr{c:hkK}a^{-\frac2\nu}=  \Cr{c:hkK}\xi$ and $rq(\Cr{c:q}r)^{-\frac2\nu}\geq cn^{\frac1\alpha}a^{-\frac1\alpha}\geq c\Cr{c:n>a}^{\frac1\alpha}a^{-\frac2\nu}\geq \Cr{C:LBLneg}a^{-\frac2\nu}$ if $n\geq \Cr{c:n>a}a^{1-\frac{2\alpha}{\nu}}$ upon fixing the constant  $\Cr{c:n>a}$ large enough. Therefore for any $K\subset\tilde{B}(\xi)$ as in \eqref{eq:K-ass}, the inequality on the right-hand side of \eqref{eq:boundParF} is satisfied. Combining this with \eqref{eq:clustCardLB}, \eqref{eq:capUB}, \eqref{eq:choicebr} and since $\sup q<\infty$, we obtain that if  ${\rm cap} ( K) \ge \Cr{C:lawcaploc}\xi^{\nu}q(\xi)^{-2}$
\begin{equation*}
\begin{split} 
\PK \big(F(K,a{\mathbf{h}}_K,r, n)\big)
    \ge c\exp\{-Ca^2r^{\nu}\}\geq c\exp\big(-C'a^{2-\frac\nu\alpha}n^{\frac\nu\alpha}\big).
\end{split}
\end{equation*}
Combining this with \eqref{eq:clustCapLocal} for $2a$ instead of $a$ and $s=\Cr{C:lawcaploc}\xi^{\nu}q(\xi)^{-2}=\Cr{C:lawcaploc}(2a\wedge1)^{-2}q(\xi)^{-2}$, which is in force since $s\leq 1\vee (2a)^{-2}$ by \eqref{eq:boundq} upon possibly slightly decreasing the constant $\Cr{C:lawcaploc}$, we thus deduce from \eqref{eq:LBcond0} that
\begin{equation}
\label{eq:lbfinal}
\begin{split} 
\P\big(|\mathcal K^{a}| 
    \ge n \big)\geq \P\big(|\mathcal K^{a}_{6\Cr{C:exit2}r}| 
    \ge n \big)&\geq 
    c\xi^{-\nu/2}\exp\big(-Ca^{2-\frac\nu\alpha}n^{\frac\nu\alpha}\big)
    \\&\geq c'n^{-\frac{\nu}{2\alpha-\nu}}\exp\big(-Ca^{2-\frac\nu\alpha}n^{\frac\nu\alpha}\big),
\end{split}
\end{equation}
where we used that $\xi^{-\nu/2}\geq ca$ and $n\geq \Cr{c:n>a}a^{1-\frac{2\alpha}{\nu}}$ in the last inequality. We have thus proved \eqref{eq:tailEstVolgen} for $n\geq \Cr{c:n>a}a^{1-\frac{2\alpha}{\nu}}$ and $a\in{(0,1]}$, and now take any $n\geq\Cr{c:n>a}$ and $a\in{[0,1]}$ such that $n< \Cr{c:n>a}a^{1-\frac{2\alpha}{\nu}}$ (with $0^{1-\frac{2\alpha}{\nu}}=\infty$). Let $a'\stackrel{\text{def.}}{=} (\Cr{c:n>a}/n)^{\frac\nu{2\alpha-\nu}}$, and note that $a'\in{(0,1]}$ and $n\geq \Cr{c:n>a}a'^{1-\frac{2\alpha}{\nu}}$. We thus just showed that $(a',n)$ satisfies \eqref{eq:tailEstVolgen}, and since $a'\geq a$ and $(a')^{2-\frac\nu\alpha}n^{\frac\nu\alpha}=c$, we deduce that
\begin{equation*}
\begin{split} 
\P\big(|\mathcal K^{a}| 
    \ge n \big)\geq \P\big(|\mathcal K^{a'}| 
    \ge n \big)\geq cn^{-\frac{\nu}{2\alpha-\nu}}\geq c'n^{-\frac{\nu}{2\alpha-\nu}}\exp\big(-Ca^{2-\frac\nu\alpha}n^{\frac\nu\alpha}\big).
\end{split}
\end{equation*}
Finally, the case $1\leq n\leq \Cr{c:n>a}$ follows from \eqref{eq:tailEstVolgen} for $n=\Cr{c:n>a}$ up to changing the constants $c,C$ therein.
\end{proof}

Let us now provide the
\begin{proof}[Proof of Lemma~\ref{lem:boundParF}]
 Let $S=\Lambda(r)\cap B(5\Cr{C:exit2}r)\setminus B(2\Cr{C:exit2}r)$, where $\Lambda(r)$ is the set from \cite[Lemma~6.1]{DrePreRod2}. Then in view of \cite[(6.1) and (6.2)]{DrePreRod2} and since $\Cr{C:exit2}\geq1$, the set $S$ satisfies
\begin{equation}
\label{eq:propertiesS}
\begin{split} 
|S|\leq C,\ B(4\Cr{C:exit2}r)\setminus B(3\Cr{C:exit2}r)\subset \bigcup_{x\in{S}}B(x,r)\text{ and } \bigcup_{x\in{S}}B(x,\Cr{C:exit2}r)\subset {B}(6\Cr{C:exit2}r)\setminus B(\Cr{C:exit2}r).
\end{split}
\end{equation}
Note  that  for any $x\in{S}$ and  $y\in{G}$, the set $\K^{-b}_{\tilde{B}(x,r)}(y)$ is connected and included in $\tilde{B}(x,r)\subset\tilde{\G}_K\cap A(a,r)$ if $r\geq \Cr{c:hkK}\xi$ by \eqref{eq:hkK} and \eqref{eq:propertiesS}. Notice further that by definition, see below \eqref{eq:hkK} and \eqref{eq:LBcond0}, for all $x\in{\tilde{B}(6\Cr{C:exit2}r)}$
\begin{equation*}
\begin{split} 
a\tilde{\mathbf{h}}_{K}(x)&=a\big(1-2P_x(H_{K}<\infty))-2P_x^K(H_{K}<\infty)\big)
\\&=a\big(1-2(P_x(H_{K}<\infty))-P_x(H_{A_{a,r}}<H_K))\big)\leq -a,
\end{split}
\end{equation*}
where in the last inequality we used that any path starting from $x$ either hit $K\subset\tilde{B}(\xi)$, or hit $A_{a,r}$ before hitting $K$, see below \eqref{eq:LBcond0}. Therefore
\begin{equation}
\label{eq:PFvstildeF}
\begin{split} 
\PK\big(F(K,a\tilde{\mathbf{h}}_K,r, n)\big)\geq\PK \big(\tilde{F}(K,a,r, n)\big),
\end{split}
\end{equation}
where 
\begin{equation}
\label{eq:AKary3}
    \tilde{F}(K,a,r, n)\stackrel{\text{def.}}{=}\left\{\begin{array}{c}
         \text{$\exists\, x\in{S},\ y\in{G}$ with $|\K^{-a}_{\tilde{B}(x,r)}(y)|\geq n$, and a}\\
    \text{continuous path }\pi\text{ in }\tilde{\G}_K\cap\tilde{B}(6\Cr{C:exit2}r)\text{ from }K\text{ to}
    \\\K^{-a}_{\tilde{B}(x,r)}(y)\text{ with }\phi_x\geq -  a \text{ for all }x\in \pi \\
    \end{array}\right\}.
\end{equation}
In order to construct the path $\pi$ in \eqref{eq:AKary3}, we are going to use the isomorphism \eqref{eq:isom} with interlacements. Writing $ u= \frac{a^2}{8}$, it indeed follows directly from \eqref{eq:isom} that under $\OPK\otimes \PK$,
\begin{equation}
\label{eq:isom3}
\begin{gathered} 
      \mathcal{I}^u\cup \big\{x\in{\tilde{\G}_K}:\phi_x\geq -a/2\big\}\text{ is stochastically dominated by }\big\{x\in{\tilde{\G}_{ K}}:\,\phi_x\geq-a\big\}.
\end{gathered}
\end{equation}
For $s>0$, let us denote by $A_s$ the event 
\begin{equation}
\label{eq:defAs}
\begin{split} 
A_s\stackrel{\text{def.}}{=}\bigcap_{x\in{S}}\Big\{\exists\, y\in G:\, \big| {\K}^{-a/2}_{\tilde B(x,r)}(y)\big| \ge n, \, \mathrm{cap}\big({\K}^{-a/2}_{\tilde B(x,r)}(y)\big)\geq sr^{\nu}\Big\}.
\end{split}
\end{equation}
On the event $A_s$ let us fix for each $x\in{S}$ an arbitrary $y_x\in{B(x,r)}$ such that $\big| {\K}^{-a/2}_{\tilde B}(y_x)\big| \ge n$ and  $\mathrm{cap}\big({\K}^{-a/2}_{\tilde B}(y_x)\big)\geq sr^{\nu}$.
Write $\mathcal I_-^{u}$ for the trajectories in $\mathcal I^{u}$ whose closure in $\tilde{\G}$ hit $K$, stopped the first time they exit $\tilde{B}(6\Cr{C:exit2}r)$. Consequently, defining on the event $A_s$ 
\begin{equation}
\label{eq:AKary}
    A'\stackrel{\text{def.}}{=}\bigcup_{x\in{S}}\Big\{\mathcal I_-^{u} \cap {\K}^{-a/2}_{\tilde B(x,r)}(y_x)\neq\emptyset\Big\},
\end{equation}
we infer using \eqref{eq:AKary3}, \eqref{eq:isom3}, \eqref{eq:defAs} and \eqref{eq:AKary} that for any $s>0$
\begin{equation} \label{eq:realizeA}
\PK \big(\tilde{F}(K,a,r, n)\big)\geq \PK\otimes\OPK(A_s\cap A')=\EK\big[1\{A_s\}\OPK(A')\big].
\end{equation}
The probability of the event $A_s$ for an appropriate choice of $s$ can be bounded using Proposition~\ref{pro:manyMeso}, and we now lower bound the probability of $A'$. It follows from a reasoning similar to the second inequality in \cite[(7.18)]{DrePreRod5} and our assumption \eqref{eq:assumpParF} that 
\begin{equation}
\label{eq:lbIu-}
\begin{split} 
\OPK\big(\mathcal{I}^u_-\cap B(3\Cr{C:exit2}r)^{\mathsf c}\cap G\neq\emptyset\big)&\geq \big(1-\exp\{-(a^2/8)\mathrm{cap}(K)\}\big)\geq ca^2\mathrm{cap}(K),
\end{split}
\end{equation}
where in the second inequality we also used the inequality $\xi^{\nu}\geq a$, and in the last one that $\mathrm{cap}(K)\leq C\xi^{\nu}=Ca^{-2}$ by \eqref{eq:capBallBd}. By \cite[(2.8)]{DrePreRod2}, the first vertex $z$ visited by $\I_-^u$ after first exiting $B(3\Cr{C:exit2}r)$ belongs to $B(4\Cr{C:exit2}r)\setminus B(3\Cr{C:exit2}r)$ if $r\geq C$, and thus belongs to $B(x,r)$ for some $x\in{S}$ by \eqref{eq:propertiesS}. Using the Markov property for the first trajectory of $\I^u_-$ which exits $B(3\Cr{C:exit2}r)^{\mathsf c}$, and recalling from below \eqref{eq:consMark} that $\Cr{C:exit2}\geq \Cr{C:exit}$, it thus follows from \eqref{eq:propertiesS}, \eqref{eq:AKary} and \eqref{eq:lbIu-} that on the event $A_s$
\begin{align}\label{eq:ALBsplit}
\begin{split}
\OPK(A') 
&\ge ca^2\mathrm{cap}(K)\inf_{x\in{S}}\inf_{z\in{B(x,r)}}P_z^K\Big(H_{{\K}^{-b/2}_{\tilde B(x,r)}(y_x)}<H_{\tilde{B}(x,\Cr{C:exit}r)^{\mathsf c}}\Big)
\\&\geq ca^2\mathrm{cap}(K)r^{-\nu}\mathrm{cap}_{\tilde{\G}_K}\big({{\K}^{-b/2}_{\tilde B(x,r)}(y_x)}\big)\geq csa^2\mathrm{cap}(K),
\end{split}
\end{align}
where we used \eqref{eq:hittingvscap} in the second inequality, which is in force in view of \eqref{eq:propertiesS} if $r\geq \Cr{c:hkK}\xi$, and the definition of ${\K}^{-b/2}_{\tilde B(x,r)}(y_x)$ in the last one.

Note that $d(x,K)\geq \Cr{C:dxK}r$ for each $x\in{S}$ by \eqref{eq:propertiesS} since $K\subset\tilde{B}(\xi)\subset \tilde{B}(r)$ and $\Cr{C:exit2}\geq\Cr{C:dxK}+1$, see  below \eqref{eq:consMark}. If $s=cq(\Cr{c:q}r)^{-2}$ and the conditions in \eqref{eq:assumpParF} are satisfied, one can thus lower bound $\PK(A_s)$ by $cq(\Cr{c:q}r)^{-C}$ in view of \eqref{eq:LBL2}, \eqref{eq:propertiesS} and the FKG inequality. Combining this with \eqref{eq:PFvstildeF}, \eqref{eq:realizeA} and \eqref{eq:ALBsplit} yields \eqref{eq:boundParF}.
\end{proof}

It remains to give the

\begin{proof}[Proof of Corollary~\ref{C:offcritvol-mom}]
 We start with \eqref{eq:offcritvol-mom2} and write, for arbitrary $k>0$, substituting $s=r^{1/k}$ below,
\begin{align}\label{eq:k-mom}
\begin{split}
\E\big[|\mathcal K^a|^k 1\{ |\mathcal K^a| < \infty\}\big]
&= \int_0^\infty \P\big( r^{1/k} \leq |\mathcal K^a| 
   <\infty \big)\, {\rm d}r\\
   &\geq k \int_1^{\infty} s^{k-1} \P\big( s \leq |\mathcal K^a| 
   <\infty \big) \, {\rm d}s.
\end{split}
\end{align}
The assumption $\sup q< \infty$ entails that Theorem~\ref{thm:volLB},(i) applies, and substituting \eqref{eq:tailEstVolgen} into \eqref{eq:k-mom} while performing the changes of variables $t=s|a|^{\frac{2\alpha}{\nu}-1} $ yields that $\E[|\mathcal K^a|^k 1\{ |\mathcal K^a| < \infty\}]$ is bounded from below for $|a| \leq 1$ by
\begin{align}\label{eq:kmom-pf2}
\begin{split}
ck \int_{|a|^{\frac{2\alpha}{\nu}-1}}^{\infty}  
\big(|a|^{1-\frac{2\alpha}{\nu} } t\big)^{k-1-\frac{\nu}{2\alpha - \nu}} e^{-Ct^{\frac\nu\alpha}}
 |a|^{1-\frac{2\alpha}{\nu}} \, {\rm d}t  \\ \geq ck |a|^{(1-\frac{2\alpha}{\nu})(k- \frac{\nu}{2\alpha - \nu})} \int_1^{\infty}  t^{k-1-\frac{\nu}{2\alpha - \nu}} e^{-Ct^{\frac\nu\alpha}} \, {\rm d}t, 
\end{split}
\end{align}
where, in obtaining the right-hand side, we have used that the exponent $\frac{2\alpha}{\nu} -1 > 0$ since $\alpha \geq \nu$ (see below \eqref{eq:intro_Green}) and the fact that $|a| \leq 1$ to derive the upper bound on the lower integration limit. The claim \eqref{eq:offcritvol-mom2} now follows for $|a| \leq 1$ since the exponent of $|a|$ appearing in \eqref{eq:kmom-pf2} can be recast as $(1-\frac{2\alpha}{\nu})(k- \frac{\nu}{2\alpha - \nu})=1-k \frac{2\alpha - \nu}{\nu}$.

It remains to show \eqref{eq:offcritvol-mom1}. One now uses the fact that $\text{cap}(K)< \infty$ implies $|K|< \infty$ under our standing assumptions on $\G$ (as follows from \cite[Lemma A.2]{DrePreRod3}; in fact, \eqref{eq:intro_Green} alone is enough since it implies in particular a uniform upper bound for the Green function on the diagonal) and subadditivity, i.e.~\eqref{eq:subadd}, to conclude that for $n\geq C$ we have
\begin{equation}
\label{eq:weakupperboundcarrd}
\begin{split} 
\P\big( n \leq |\mathcal K^a| 
   <\infty \big) \geq \P\big( Cn \leq \text{cap}(\mathcal K^a) 
   <\infty \big)\geq cn^{-\frac12}e^{-Ca^2n} ,
\end{split}
\end{equation}
where the last inequality follows from \eqref{eq:capTail}.
Feeding this into \eqref{eq:k-mom} and performing the change of variables $t=|a|^2 s$ thus yields that 
$$
\E\big[|\mathcal K^a|^k 1\{ |\mathcal K^a| < \infty\}\big] \geq ck  \int_C^{\infty} \big( |a|^{-2}t\big)^{k-1-\frac12} e^{-Ct}  |a|^{-2}\, {\rm d}t ,
$$  
for all $|a|\leq 1$, from which \eqref{eq:offcritvol-mom1} readily follows.
\end{proof}

\begin{Rk}
\phantomsection\label{rk:final}
\begin{enumerate}[1)]
    \item\label{rk:optimalboundmeanfield}
 As explained below Corollary~\ref{C:offcritvol-mom}, \eqref{eq:offcritvol-mom1} is expected to be sharp in the mean-field regime $2\alpha<3\nu$. Since it is a direct consequence of the simple lower bound \eqref{eq:weakupperboundcarrd}, and that this lower bound is actually sharp when $a=0$ in this regime on the square lattice by \cite[Theorem~1.2]{cai2023onearm}, one may be tempted to conclude that \eqref{eq:weakupperboundcarrd} is also sharp in the mean-field regime. However, Theorem~\ref{thm:volLB}, (ii) shows that this is not the case and that $\P(|\K^a|\geq n)$ in fact decays essentially as $\exp\{-cn^{\frac\nu\alpha}\}$ for $c=c(a)$. We conjecture that the correct decay when $2\alpha<3\nu$ is
\begin{equation}
\label{eq:conjecturemeanfield}
\begin{split} 
cn^{-\frac12}\exp\big\{-Ca^\frac{2\nu}{\alpha}n^{\frac\nu\alpha}\big\}\leq \P\big(|\K^a|\geq n\big)\leq Cn^{-\frac12}\exp\big\{-ca^\frac{2\nu}{\alpha}n^{\frac\nu\alpha}\big\}. 
\end{split}
\end{equation}
for all $a\in{[0,1]}$ and $n\geq1$. Note that the conjecture \eqref{eq:conjecturemeanfield} would also directly yield \eqref{eq:offcritvol-mom1} as well as a matching upper bound, that it is compatible with \eqref{eq:easyoffcriticalbounds}, and that when $2\alpha=3\nu$ it agrees with the bound \eqref{eq:tailEstVolgen} (conjectured optimal up to constants when $2\alpha>3\nu$). Since \eqref{eq:conjecturemeanfield} differs for $a\neq0$ from the case of independent percolation as explained below \eqref{eq:offcritvol-n}, the proof of \eqref{eq:conjecturemeanfield} would probably require to use tools different from those of independent percolation, contrary to the the critical case $a=0$ from \cite{cai2023onearm}.

\item\label{rk:boundinanydimension} Even if $\sup q=\infty$, it is possible to prove that \eqref{eq:tailEstVolgen}, resp.\ \eqref{eq:tailEstlargestVolgen}, for $a\neq0$ remains true under the additional assumption $n\geq C|a|^{-C'}$, resp.\ $r\geq C|a|^{-C'}$, where $C'$ is a large constant. The proof is exactly the same as the proof of \eqref{eq:tailEstVolgen}, resp.\ \eqref{eq:tailEstlargestVolgen}, except one replaces the use of \eqref{eq:LBL2} by \eqref{eq:LBL4}, and replaces $\xi=a^{-\frac2\nu}$ by $\xi=a^{-C}$ for a large enough constant $C$ so that ${\rm cap} ( \K_\xi^{2a}) \ge a^{-2/\nu}$ with probability at least $a^{-2}$ using \eqref{eq:boundq} and \eqref{eq:clustCapLocal}, which allows us remove the dependency on $q$ in \eqref{eq:lbIu-} and then \eqref{eq:boundParF}. We leave the details to the reader. Note that in view of the conjecture \eqref{eq:conjecturemeanfield}, one actually expects \eqref{eq:tailEstVolgen} to be true for all $n\geq1$ and $a\in{[-1,1]}$ on any graph satisfying \eqref{eq:ellipticity}, \eqref{eq:intro_sizeball} and \eqref{eq:intro_Green}. This is however not the case for \eqref{eq:tailEstlargestVolgen} since by \cite[Proposition~3]{werner2020clusters}, $\LV_r^0$ is smaller than  $r^4\log(r)= o(r^{\frac{d+2}{2}})$(=$o(r^{\alpha-\frac\nu2})$) with probability going to one on $\Z^d$ for all $d\geq7$. In particular, \eqref{eq:LBL4} is not satisfied for all $r\geq Ca^{-2/\nu}$ in dimension $d\geq7$, since otherwise it would imply \eqref{eq:tailEstlargestVolgen} by proceeding similarly as in the proof of Theorem~\ref{the:mainMra}.
 \item \label{rk:d=6}
When $\sup q=\infty$, it is possible to extend \eqref{eq:tailEstVolgen} but with an additional dependency on $q$ in the exponential. For instance on $\Z^6$, it was proved \cite{cai2024onearm} that $q(r)\leq \tilde{q}(r)$, where $\tilde{q}(r)=1 \vee \exp\{C \sqrt{\log r} \log \log r\}$ is subpolynomial. One then obtains that, up to changing the constant $C$ in the definition of $\tilde{q}(r)$, for all $a\in{[-1,1]\setminus\{0\}}$ and $n\geq C|a|^{-2}\tilde{q}(|a|^{-1})$,
\begin{equation} \label{eq:offcritvol-Z6-LB}
\P\big( n \leq |\mathcal K^a| <\infty \big) \ge
 c\exp \big\{ -C |a|^{\frac{8}{6}}n^\frac{4}{6} \tilde{q}(n/a)\big\}.  
\end{equation}
The exponential correction in \eqref{eq:offcritvol-Z6-LB} is the same as in lower dimension \eqref{eq:offcritvol-Zd-LB}, but with additional subpolynomial corrections. One actually expects that $q(r)\leq \log(r)^C$, which if proved would improve \eqref{eq:offcritvol-Z6-LB} by replacing $\tilde{q}(n/a)$ by $\log(n/a)^{C'}$ therein. Note that the polynomial correction in \eqref{eq:offcritvol-Zd-LB}, corresponding to the critical factor, was absorbed in the exponential of \eqref{eq:offcritvol-Z6-LB} since $\exp\{-\tilde{q}(n/a)\}$ is always super-polynomial in $n$ under our conditions. In particular, one cannot use our proof to obtain the lower bound \eqref{eq:critvol-Zd} at criticality $a=0$, but the same bound actually already follows in dimension six from the law of the capacity, see \eqref{eq:critvol-Z6}. We leave the proof of \eqref{eq:offcritvol-Z6-LB} to the reader, and only indicate that one needs to adapt the definition of $r$ in \eqref{eq:choicebr} so that the last conditions appearing in \eqref{eq:assumpParF} and \eqref{eq:boundParF} are still trivially verified under the condition $q(r)\leq \tilde{q}(r)$. A lower bound similar to \eqref{eq:tailEstlargestVolgen} for $r\geq Ca^{-\frac2\nu}\tilde{q}(a^{-1})$ and $a>0$, but adding $\tilde{q}(r)$ in the exponential,  can also be similarly obtained. Note that by the previous remark, one can in fact remove the factor $\tilde{q}(n/a)$ in \eqref{eq:offcritvol-Z6-LB} if $n\geq C|a|^{-C'}$, but for a constant $C'$ a priori much larger than $2$.

\item \label{rk:volumeinball} As the attentive reader might have noticed, see for instance \eqref{eq:choicebr} and \eqref{eq:lbfinal}, the lower bound  \eqref{eq:tailEstVolgen} is actually still valid when replacing $|\K^a|$ by $|\K_r^a|$ for any $r\geq C\big((n/|a|)^{\frac1\alpha}\wedge n^{\frac2{2\alpha-\nu}}\big)$, and similarly for \eqref{eq:easyoffcriticalbounds} if $r\geq Cn^{\frac1\alpha}$. In particular taking $a=0$ and summing in $n$ we deduce that if $\sup q<\infty$, then for all $k\geq1$
\begin{equation}
\label{eq:momentKra}
\begin{split} 
\E[|\K^0_r|^k]\geq cr^{\frac{(2\alpha-\nu)k-\nu}2}
\end{split}
\end{equation}
for some constant $c=c(k)$. The lower bound \eqref{eq:momentKra} is sharp up to constants for $k=1$ by \cite[Proposition~5.2]{MR3502602}, and we expect that this is actually the case for all $k\geq1$ if $\sup q<\infty$.

Moreover, one can deduce from the previous observation a lower bound similar to \eqref{eq:tailEstVolgen} or \eqref{eq:easyoffcriticalbounds} when replacing $|\K^a|$ therein by $\LV_r^a$ for $r$ as before. By a union bound, the upper bound in \eqref{eq:easyoffcriticalbounds} also holds when replacing $|\K^a|$ therein by $\LV_r^a$, but with an additional logarithmic factor in $r$, which can be removed whenever $n\geq Ca^{-\frac{2\alpha}\nu}\log(r)^{\frac\alpha\nu}.$

\item \label{rk:wernerconjecture} The conjecture from \cite[Conjecture~A and C]{werner2020clusters} postulates that on $\Z^d$ for $d\in{\{3,4,5\}}$, the set $\{\phi\geq0\}$ has a scaling limit, that the clusters in the scaling limit have fractal (or Hausdorff) dimension $(d+2)/2$, and that there are only finitely many such clusters in a finite domain with diameter greater than $s$ for any $s>0$. Our proof can be adapted to show that the average number of clusters with diameter larger than $s$ and fractal dimension larger than $(d+2)/2$ is of constant order for $s$ small enough, which is a first step toward the previous conjecture. More precisely, the corresponding result in the discrete setting and on general graph is that if $\sup q<\infty$, then for any $0<s\leq c$, there exist $c,C(s)\in{(0,\infty)}$ so that for all $r\geq1$ we have
\begin{equation}
\label{eq:numberbigradiusandvolume}
\begin{split} 
c\leq \E\Big[\big|\big\{K 
\, : \, & \text{$K$ cluster of }\{x\in{\tilde{B}(r)}:\phi_x\geq0\} \text{ such } \\
&\text{ that } |K|\in{[cr^{\alpha-\frac{\nu}2},Cr^{\alpha-\frac\nu2}]},\delta(K)\geq sr\big\}\big|\Big]\leq C,
\end{split}
\end{equation}
where $\delta(K)=\sup\{d(x,y):\,x,y\in{K\cap G}\}$ is the diameter of a set $K$. In order to obtain \eqref{eq:numberbigradiusandvolume}, one simply notices that the set $\hat{\I}^u_{\tilde B}\cup\tilde{\mathcal H}(x,r, u^{-1}, \hat{\I}^u_{\tilde B})$ constructed below \eqref{eq:tildeH} has diameter at least $\geq cr$ since it contains interlacements trajectories started in $B(x,r)$ and ending in $B(x,\Cr{C:exit}r)^{\mathsf c}$, and so one deduces a version of \eqref{eq:LBL2} with an additional event $\delta(\K^{-a}_{\tilde{B}}(y))\geq cr$ inside the probability. By the change of measure \eqref{eq:entropy}, see the proof of \eqref{eq:tailEstlargestVolgen} for details, one can then deduce that with constant probability there exists an $x\in{B(r)}$ so that $\K_r^0(x)$ has cardinal at least $cr^{\alpha-\frac\nu2}$ and diameter at least $cr$. Combining this with Proposition~\ref{pro:upperMRa} for $\eps$ small enough, we obtain
\begin{equation}
\label{eq:bigradiusandvolume}
\begin{split} 
\P\big(\exists\, x\in{B(r)}:|\K^0_r(x)|\in{[cr^{\alpha-\frac{\nu}2},Cr^{\alpha-\frac\nu2}]},\delta(\K^0_r(x))\geq cr\big)\geq c,
\end{split}
\end{equation}
Note that \eqref{eq:bigradiusandvolume} directly implies the lower bound in \eqref{eq:numberbigradiusandvolume} for $s\leq c$. For the upper bound, one simply notices that the probability in \eqref{eq:numberbigradiusandvolume} is smaller than $Cr^{-(\alpha-\frac\nu2)}$ times the average number of $x\in{B(r)}$ with $\delta(\K^0(x))\geq sr$, which is smaller than a constant (depending on $s$) if $\sup q<\infty$ by \eqref{eq:intro_sizeball}, \eqref{eq:lambdabounded} and \eqref{eq:qPsiBd}.
\end{enumerate}
\end{Rk}

\appendix
\setcounter{secnumdepth}{0}
\section{Appendix: Lower bound in the mean-field regime}
\renewcommand*{\theThe}{A.\arabic{The}}
\setcounter{The}{0}
\setcounter{equation}{0}
\renewcommand{\theequation}{A.\arabic{equation}}

Recall the definition of $q(\cdot)$ from \eqref{eq:qPsiBd}. We derive here a lower bound on $M_r^0$, which is expected to be sharp up to logarithm in the mean-field regime $2\alpha\leq 3\nu$, and whose proof follow ideas from Bernoulli percolation \cite[Thm. 13.22]{HeHo-17}. 

\begin{Prop}
\label{pro:meanfieldlower}
For all $r,t\geq1$ with $q(r/2)^{-2}r^{\nu}\geq Ct$
\begin{equation}
\label{eq:meanfieldlower}
\begin{split} 
\P\big(\LV^0_r\geq (1/t)q(r/2)^{-2}r^{\nu}\big)
            \geq 1- Ct^{-1}.
\end{split}
\end{equation}
\end{Prop}
Note that on $\Z^d$, $d\geq7$, we have \cite{cai2023onearm} that $q(r/2)^{-2}r^{\nu}\geq cr^{4}$, and so \eqref{eq:meanfieldlower} complements \cite[Proposition~3]{werner2020clusters}.  
\begin{proof}
First note that for any $x\in{B(r/2)}$ by subadditivity of capacity \eqref{eq:subadd}
\begin{equation*}
\begin{split} 
\P\big(|\K^0_{r/2}(x)|\geq (1/t)q(r/2)^{-2}r^{\nu}\big)&\geq \P(\mathrm{cap}(\K^0(x))\geq csq(r/2)^{-2}r^{\nu})-\P\big(x \stackrel{\geq 0}{\longleftrightarrow} \partial \widetilde{B}(x,r/2)\big)
\\&\geq ct^{\frac12}q(r/2)r^{-\nu/2}-q(r/2)(r/2)^{-\nu/2}\geq ct^{\frac12}q(r/2)r^{-\nu/2},
\end{split}
\end{equation*}
where the second inequality follows from \eqref{eq:capTail} and \eqref{eq:qPsiBd}, and the last inequality holds upon taking $t\geq\Cl{c:slbmeanfield}$ for a small enough constant $\Cr{c:slbmeanfield}>0$. In particular by \eqref{eq:intro_sizeball} and \eqref{eq:lambdabounded}, for all $t\geq \Cr{c:slbmeanfield}$ and $r\geq1$
\begin{equation}
\label{eq:firstmomentZr}
\begin{split} 
\E[Z_r]\geq ct^{\frac12}q(r/2)r^{\alpha-\nu/2},\text{ where }Z_r=\sum_{x\in{B(r)}}1\{|\K^0_r(x)|\geq s_r \}\text{ and }s_r=(1/t)q(r/2)^{-2}r^{\nu}.
\end{split}
\end{equation}
We will now use a second moment method to show that $Z_r$ is positive with large probability. First note that for all $x,y\in{B(r)}$, by the Markov property \eqref{eq:Markov2} with $K=\emptyset$ and $\K=\K^0(x)$, we have on the event $y\notin{\K^0(x)}$
\begin{equation*}
\begin{split} 
\P\big(|\K^0_r(y)|\geq s_r\,|\,\mathcal{A}_{\K^0(x)}^+\big)&=\PKK{\K^0(x)}\big(|\K^0_r(y)|\geq s_r\big)
\leq \P\big(|\K^0_r(y)|\geq s_r\big),
\end{split}
\end{equation*}
where the last inequality follows from the fact that the law of $\K^0_r(y)$ under $\PKK{\K^0(x)}$ is stochastically dominated by the law of $\K^0_r(y)$ under $\P$ by a similar argument as below \eqref{eq:capinball}. Noting that $\K^0_r(x)$ is $\mathcal{A}_{\K^0(x)}^+$ measurable in view of \eqref{eq:Markov1} and \eqref{eq:Kr}, we readily deduce that
\begin{equation*}
\begin{split} 
\P\big(|\K^0_r(x)|\geq s_r,|\K^0_r(y)|\geq s_r,y\notin{\K^0(x)}\big)
\leq \P\big(|\K^0_r(x)|\geq s_r\big)\P\big(|\K^0_r(y)|\geq s_r\big).
\end{split}
\end{equation*}
Since $\P(y\in{\K^0(x)})\leq Cd(x,y)^{-\nu}$ by \cite[Prop. 5.2]{MR3502602} and \eqref{eq:intro_Green}, we readily deduce that
\begin{equation}
\label{eq:secondmomentZr}
\begin{split} 
\text{Var}(Z_r)\leq C\sum_{x,y\in{B(r)}}(1\vee d(x,y))^{-\nu}\leq Cr^{\alpha}\sum_{k=0}^{\lceil\log_2(2r)\rceil}2^{k\alpha-k\nu}\leq Cr^{2\alpha-\nu},
\end{split}
\end{equation}
where the second inequality follows from \eqref{eq:intro_sizeball} and \eqref{eq:lambdabounded} after decomposing the ball $B(x,2r)$ in dyadic scales. For $t\geq\Cr{c:slbmeanfield}$, combining \eqref{eq:firstmomentZr}, \eqref{eq:secondmomentZr} and Chebyshev's inequality, we have
\begin{equation}
\label{eq:finalboundZr}
\begin{split} 
\P\big(Z_r\leq ct^{\frac12}q(r/2)r^{\alpha-\nu/2}\big)\leq \frac{Cr^{2\alpha-\nu}}{(t^{\frac12}q(r/2)r^{\alpha-\nu/2})^2}\leq Ct^{-1}.
\end{split}
\end{equation}
in view of \eqref{eq:boundq}. In particular, if $t\geq\Cr{c:slbmeanfield}$ and $q(r/2)r^{\alpha-\nu/2}\geq Ct^{\frac{1}{2}}$,  $Z_r\geq 1$ with  probability at least $1-Ct^{-1}$, which implies \eqref{eq:meanfieldlower} in view of \eqref{eq:defVra} and \eqref{eq:Kr}. Note further that the condition $\ q(r/2)^{-2}r^{\nu}\geq Ct$ implies that $q(r/2)r^{\alpha-\nu/2}\geq Ct^{\frac{2\alpha-\nu}{2\nu}}\geq Ct^{\frac{1}{2}}$, where we used \eqref{eq:boundq} as well as our assumption $\alpha\geq \nu+2$, see below \eqref{eq:intro_Green}. Noting finally that \eqref{eq:meanfieldlower} for $t\leq\Cr{c:slbmeanfield}$ is trivial up to changing the constant $C$ therein, we can conclude.
\end{proof}

As the attentive reader will have noticed, on $\Z^d$ for all $d\geq7$, the above proof (see in particular \eqref{eq:finalboundZr}) combined with \cite{cai2023onearm} actually implies that, with probability at least $1-Cr^{6-d}$, there are at least $cr^{d-2}$ vertices which belong to a critical cluster with volume at least $cr^4$ in $B(r)$. In particular by \cite[Proposition~3]{werner2020clusters}, there are at least  $cr^{d-6}/\log(r)$ critical clusters with volume greater than $cr^4$ in $B(r)$ with probability going to one as $r\rightarrow\infty$, which corresponds to \cite[Proposition~4]{werner2020clusters} but for volume instead of radius.

\medskip

{\bf Acknowledgment:}
This research was supported through the programs \lq Oberwolfach Research Fellows\rq\ and \lq Oberwolfach Leibniz Fellows\rq\ by the Mathematisches Forschungsinstitut Oberwolfach in 2023. The authors would like to thank the Isaac Newton Institute for Mathematical Sciences, Cambridge, for support and hospitality during the program Stochastic systems for anomalous diffusion, where work on this paper was undertaken. This work was supported by EPSRC grant EP/Z000580/1. The research of AD has been supported by the Deutsche Forschungsgemeinschaft (DFG) grant DR 1096/2-1. AP has also been supported by the Swiss NSF, the Isaac Newton Trust grant G101121 \lq Interplay of random media and statistical mechanics\rq\ and the Engineering and Physical Sciences Research Council grant EP/R022615/1 \lq Random walks on dynamic graphs.\rq

\bibliography{bibliographie}
\bibliographystyle{abbrv}

\end{document}